\documentclass[1p,preprint]{elsarticle}
\journal{Linear Algebra and its Applications}
\usepackage[dvipsnames]{xcolor}
\usepackage{algpseudocode}
\usepackage{algorithm}
\usepackage{amsmath}
\usepackage{amssymb}
\usepackage{amsthm}
\usepackage{comment}
\usepackage[subnum]{cases}
\usepackage{subfigure}
\usepackage{pgfplots}
\usepackage{tikz}
\usetikzlibrary{decorations.pathreplacing}
\usetikzlibrary{patterns,arrows,decorations.pathreplacing}
\usetikzlibrary{shapes.misc}

\title{Min-Max Elementwise Backward Error\\
for Roots of Polynomials and a\\
Corresponding Backward Stable Root Finder\tnoteref{t1}}
\tnotetext[t1]{Version of January 14, 2020.}

\makeatletter
\def\ps@pprintTitle{%
     \let\@oddhead\@empty
     \let\@evenhead\@empty
     \def\@oddfoot{}
     \let\@evenfoot\@oddfoot}
\makeatother

\makeatletter
\@addtoreset{equation}{section}
\@addtoreset{figure}{section}
\@addtoreset{table}{section}
\makeatother

\newtheorem{example}{Example}

\def\deg{d}

\def\twobyone#1#2{\bigl[{\hfil#1\atop\hfil#2}\bigr]}

\def\qedsymbol{\vbox{\hrule\hbox{%
                     \vrule height1.3ex\hskip0.8ex\vrule}\hrule}}

\def\myendproof{\qquad\qedsymbol}
\def\noqed{\def\qedsymbol{}}
\def\mystrut#1{\rule{0cm}{#1}}
\def\C{\mathbb{C}}
\def\R{\mathbb{R}}

\newcommand{\hz}{\widehat{z}}

\newcommand{\tp}{\widetilde{p}}
\newcommand{\delp}{\Delta p}
\newcommand{\delbp}{\Delta \bp}
\newcommand{\relp}{\delta p}
\newcommand{\delz}{\Delta z}
\newcommand{\hbeta}{\widehat{\beta}}
\newcommand{\hgamma}{\widehat{\gamma}}
\newcommand{\hA}{\widehat{A}}
\newcommand{\hB}{\widehat{B}}
\newcommand{\wA}{\widetilde{A}}
\newcommand{\wB}{\widetilde{B}}
\newcommand{\dhA}{\Delta\widehat{A}}
\newcommand{\dha}{\Delta\widehat{a}}
\newcommand{\dhB}{\Delta\widehat{B}}
\newcommand{\ha}{\widehat{a}}
\newcommand{\hb}{\widehat{b}}
\newcommand{\be}{\eta}
\newcommand{\eps}{\epsilon}
\newcommand{\OO}{{\cal O}}

\newcommand{\emach}{\epsilon_{\mbox{{\scriptsize mach}}}}

\newcommand{\tr}{\tau}
\newcommand{\ttr}{\widetilde{\tr}}
\newcommand{\Rp}{\mathbb{R}^+}
\newcommand{\Rmaxt}{\mathbb{R}_{\max,\times}}
\newcommand{\tplus}{\oplus}
\newcommand{\tmult}{\otimes}
\newcommand{\tpol}{{\tt t}p}
\newcommand{\hC}{\widehat{C}}
\newcommand{\diag}{\mathrm{diag}}
\newcommand{\epsmach}{\epsilon_{\mbox{\scriptsize mach}}}

\newcommand{\cred}[1]{{\color{red} #1}}

\newcommand{\MATLAB}{{MATLAB}}
\newcommand{\CC}{\mathbb C}
\newcommand{\polyeig}{{\tt polyeig }}
\newcommand{\quadeig}{{\tt quadeig }}

\newcommand{\ftt}[1]{{\tt #1}}
\newcommand{\ftts}[1]{{\small\tt #1}}
\newcommand{\ttb}[1]{\ftts{#1}}

\newcommand{\na}{---}
\newcommand{\tc}[1]{\bf\cred{#1}}

\newcommand{\bp}{{\bold p}}
\newcommand{\btp}{\widetilde{\bold p}}
\newcommand{\berra}{\eta^{\mathrm{norm}}}
\newcommand{\berrb}{{\eta^{\mathrm{elem}}_{|\bp|}}}
\newcommand{\berrc}{{\eta^{\mathrm{elem}}_{\btg}}}

\newtheorem{theorem}{Theorem}
\newtheorem{assumption}{Assumption}
\newtheorem{definition}{Definition}

\newenvironment{ft}{\begin{quote}\color{blue} \small\sf FT $\diamondsuit$~}{\end{quote}}

\newcounter{exemple}
\def\Experiment{Experiment}

\newenvironment{exemple}%
               {\begin{list}{\indent\emph{\Experiment}
               {\upshape\arabic{exemple}.}}%
                {\usecounter{exemple}
                \setlength{\leftmargin}{\rightmargin}
                \setlength{\labelwidth}{\leftmargin}
                \addtolength{\labelwidth}{-\labelsep}
                \setlength{\topsep}{0in}
                \setlength{\itemsep}{0pt}
                \setlength{\listparindent}{\parindent}%
               }}%
               {\end{list}}
\def\l{\lambda}
\def\bz{\bold z}
\def\bhz{\widehat{\bz}}
\def\ba{\boldsymbol\alpha}
\def\etaea{\eta^{\mathrm{elem}}_{\ba}}
\def\btg{\widetilde{\boldsymbol\gamma}}
\def\err{\mathrm{err}}

\begin{document}
\begin{frontmatter}

\author[man]{Fran\c{c}oise Tisseur}\ead{francoise.tisseur@manchester.ac.uk}
\author[ku]{Marc Van Barel\corref{cor1}\fnref{label2}}\ead{marc.vanbarel@cs.kuleuven.be}

\address[man]{Department of Mathematics, The University of Manchester,
              Manchester, M13 9PL, UK}
\address[ku]{Department of Computer Science, KU Leuven, B-3001 Leuven (Heverlee),
Belgium}
\cortext[cor1]{Corresponding author.}
\fntext[label2]{This author was partially supported by
the Research Council KU Leuven,
C1-project (Numerical Linear Algebra and Polynomial Computations),
and by
the Fund for Scientific Research--Flanders (Belgium),
G.0828.14N (Multivariate polynomial and rational interpolation and approximation),
and EOS Project no 30468160.}

\begin{abstract}
A new measure called min-max elementwise backward error is introduced
for approximate roots of scalar polynomials $p(z)$.
Compared with the elementwise relative backward error, this new measure allows for
larger relative perturbations on the coefficients of $p(z)$ that do not participate much
in the overall backward error.
By how much these coefficients can be perturbed is determined
via an associated max-times polynomial and its tropical roots.
An algorithm is designed for computing the roots of $p(z)$.
It uses a companion linearization $C(z) = A-zB$ of $p(z)$ to which we
added an extra zero leading coefficient, and
an appropriate two-sided diagonal scaling that balances $A$ and makes $B$
graded in particular when there is variation in the magnitude
of the coefficients of $p(z)$.
An implementation of the QZ algorithm with a strict deflation criterion for
eigenvalues at infinity is then used to obtain approximations to the roots of $p(z)$.
Under the assumption that this implementation of the QZ algorithm exhibits a
graded backward error when $B$ is graded,
we prove that our new
algorithm is min-max elementwise backward stable.
Several numerical experiments show the superior performance of the new algorithm
compared with the MATLAB \texttt{roots} function.
Extending the algorithm to polynomial eigenvalue problems leads to
a new polynomial eigensolver
that exhibits excellent numerical behaviour compared with other existing
polynomial eigensolvers, as illustrated by many numerical tests.
\end{abstract}

\begin{keyword}
zeros of polynomials \sep
polynomial root finder \sep
elementwise backward error\sep
tropical roots\sep
polynomial eigenvalue problems\sep
(block) companion linearization\sep

\MSC
65F15\sep 65H04\sep 30C15\sep 15A22\sep 15A80\sep 15A18\sep 47J10
\end{keyword}

\end{frontmatter}

\section{Introduction}\label{sec_intro}
We consider the problem of computing all the zeros $z_k$, $k = 1,2,\ldots,\deg$, of
the scalar polynomial  $p(z)$ of degree $\deg$ expressed in the monomial basis
as
$$
p(z) = \sum_{i=0}^\deg p_i z^i.
$$
We denote by $\hz_k$ the approximate zeros computed by some algorithm whose
numerical stability we  want to assess.
For this we consider the backward error in a \emph{global} way, i.e., for
all computed roots at the same time. The computed zeros $\hz_k$ are the exact
zeros of a polynomial
$$
        \tp(z) = \sum_{i=0}^\deg \tp_i z^i
        = \tp_d(z-\hz_1)\cdots(z-\hz_d)=p(z)+\Delta p(z)
$$
and the backward error
measures the difference between the vector of coefficients $\bp = \left[ p_0,
p_1,\ldots,p_\deg \right]$ of the given polynomial $p(z)$ and the
vector of coefficients $\btp = \left[ \tp_0, \tp_1,\ldots, \tp_\deg \right]$ of the polynomial
$\tp(z)$.
Assuming that $p_d=\tp_d$, we can consider the normwise relative backward error
$$
\berra = \frac{\| \btp - \bp \|}{\| \bp \|}= \frac{\|\Delta \bp\|}{\|\bp\|}
$$
for some vector norm $\|\cdot\|$
or the elementwise relative backward error
\begin{equation}\label{eq.releleberr}
\berrb = \max_{i, p_i \neq 0} \frac{| \tp_i - p_i |}{| p_i |}
 \quad \mbox{if $\tp_i = 0$ whenever  $p_i = 0$,}
\end{equation}
and $\berrb = \infty$ if $\tp_i\ne 0$ when $p_i=0$ for some $i$.
This elementwise backward error was studied in~\cite{EdelmanMurakami1995}
expanding on earlier work by Van Dooren and Dewilde~\cite{n643}.
In fact, with the latter measure of the backward error, there is
no backward stable polynomial root solver~\cite{MVD2015}
while there exist several normwise backward stable algorithms, e.g., the fast polynomial root solver
described in~\cite{AMVW2015}.
Example~\ref{ex0} below shows that the normwise backward error $\berra$
can be much smaller than the elementwise relative backward error
$\berrb$ and when combined with a
condition number, they do not provide sharp
upper bounds on the relative errors $|z_i-\hz_i|/|\hz_i|$.
So we introduce in section~\ref{sec_def_be} a new measure of the
backward error, denoted by $\eta^{\mathrm{elem}}_{\btg}$, called min-max
elementwise backward error, and for which the perturbations
$\Delta p_i$ are measured
relative to some parameters $\widetilde\gamma_i\ge |p_i|$.
In section~\ref{sec_compute_gamma}, we show the connection between the
parameters
$\btg=[\widetilde\gamma_0,\widetilde\gamma_1,\ldots,\widetilde\gamma_\deg]$
associated with this new backward error measure and the tropical roots of the
max-times polynomial $\tpol(x)= \max_i (|p_i| x^i)$ associated with $p(z)$.
In section~\ref{sec_algo}, we describe a new polynomial root finder
for $p(z)$
based on a $(\deg+1)\times (\deg+1)$ companion linearization
$C(z) = A-zB$ of the grade $\deg+1$ polynomial
$0z^{\deg+1}+p(z)$ and an appropriate two-sided diagonal scaling of $C(z)$ that
balances the matrix $A$ and makes the matrix $B$ graded when there are large
variation in the magnitude of the tropical roots.
This property of the scaled pencil is crucial for the numerical stability of
our algorithm and can be difficult to achieve on a companion
linearization of $p(z)$ but not for $0z^{\deg+1}+p(z)$.
The diagonal scaling of $C(z)$ is then followed by a deflation of the
artificially introduced eigenvalue at infinity.
Finally, we use an implementation of the QZ algorithm with a ``strict" deflation
criterion for the eigenvalues at infinity to compute the finite eigenvalues of
the scaled and deflated pencil, which we return as approximate roots of $p(z)$.
We prove in section~\ref{sec_be} that this new polynomial root finder is min-max elementwise backward
stable under the assumption that, when applied to a pencil $A- zB$ with $A$
well-balanced and $B$ graded, the QZ algorithm with strict deflation at infinity
computes the exact generalized Schur form of a perturbed pencil
$A+\Delta A- z(B+\Delta B)$ with
$|(\Delta A)_{ij}|$ of order of the machine precision $\emach$
and a $\Delta B$ that can be
written as $\emach$ times a graded matrix.
Section~\ref{sec_num_exp} presents numerical experiments that illustrate
the min-max elementwise backward stability of the new polynomial root finder.
We explain how to extend our algorithm
to the computation of eigenvalues of matrix polynomials.
This leads to a new polynomial eigensolver based on a tropically
scaled block companion pencil. Numerical experiments
show that this new polynomial eigensolver and the eigensolver based on a
tropically scaled Lagrange linearization described in~\cite{VBT2017}
both compute eigenvalues with small relative normwise backward errors.
An advantage of the new eigensolver over that in~\cite{VBT2017}, is that it does
not require the computation of ``well-separated tropical roots" and is easier to
implement.
Section~\ref{sec_concl} gives our conclusions.

\begin{example}\label{ex0}
Let us now compute the three measures $\berra$, $\berrb$ and
$\eta^{\mathrm{elem}}_{\btg}$ for the backward error
when computing the roots of
\begin{equation}\label{def.ex1}
        p(z) = z^4-z^3+2\cdot 10^{-25}z^2 +10^{-30}z-10^{-60}
\end{equation}
using the {\rm MATLAB} function {\rm\texttt{roots}} and
the new algorithm {\rm(}written in {\rm MATLAB}{\rm)}. The results are provided in Table~\ref{tab1} together with the
relative forward errors
$$
      \err(\hz_k):=\frac{|z_k-\hz_k|}{|z_k|},
$$
$k=1,\ldots,4$.
For this example, the new algorithm computes the roots of $p(z)$ more accurately
than {\rm\texttt{roots}}.
Note that we could wrongly decide that the roots have been well computed when
looking at $\berra$ when using  {\rm\tt roots}.

Assuming that the zeros $z_k$ of $p(z)$
are all simple and neglecting the higher order terms,
the relative forward error can be written as
\begin{equation}\label{eq771}
\err(z_k) = \frac{|\delp(\hz_k)|}{|z_k|\,| p'(\hz_k)|}.
\end{equation}
It will be clear from section~\ref{sec_def_be} that the
following upper bounds for the numerator in \eqref{eq771} hold:
\begin{numcases}
{| \delp(\hz_k) |\leq}
\berra \| [p_0,p_1, \ldots,p_d]\| \,\,\| 1, \hz_k,
\ldots,\hz_k^d] \|, \label{ubound1} \\
\berrb \sum_{i=0, p_i \neq 0}^d |p_i| |\hz_k|^i, \label{ubound2} \\
\berrc (d+1) \max_{0\leq j \leq d} | p_j \hz_k^j | . \label{ubound3}
 \end{numcases}
%
Note that the small roots of $p(z)$ in \eqref{def.ex1} are ill conditioned for a
normwise measure of the perturbations since $\| [p_0,p_1, \ldots,p_4]\| \,\,\| 1,
\hz_k, \ldots,\hz_k^4] \|$ is of order one but $|z_k|\,| p'(\hz_k)|$ is very
small.
On the other hand, all the roots of $p(z)$ are all well-conditioned when
perturbations are measured elementwise, i.e., the values of $(\sum_{i=0, p_i \neq
0}^4 |p_i| |\hz_k|^i)/(|\hz_kp'(\hz_k)|$ and $\max_{0\leq j \leq 4} | p_j
\hz_k^j |/(|\hz_kp'(\hz_k)|$ are of order 1.
From~\eqref{ubound3}, it is clear that the computed
zeros are the exact zeros of a polynomial with a relative error of the order of
$\berrc$ on the coefficients of the \emph{dominant} terms. Hence, the zeros
computed by the new algorithm can be seen as the exact zeros of a polynomial
whose coefficients were rounded up to the order of the machine precision
since $\berrc =6.7\times 10^{-16}\approx 3\emach$.
This is the best we can hope for when computing in finite precision.  This
results in computed zeros with relative forward errors of the size of the
machine precision.

\begin{table}
\caption{Relative error $\err(\hz_k)$ for the four roots of $p(z)$ in
\eqref{def.ex1} computed by \texttt{roots} and by the new root solver, and
corresponding backward errors.}\label{tab1}
\begin{center}
\footnotesize
\begin{tabular}{ccccc}
&& \texttt{roots} && new algorithm\\ \hline
\mystrut{.4cm}
$z_k$ && $\err(\hz_k)$ && $\err(\hz_k)$\\
\ftt{-9.999999999000001e-16} &&\ftt{1.5e-09}&&\ftt{1.1e-16} \\
\ftt{+9.999999999999999e-31} &&\ftt{5.1e-02}&&\ftt{1.5e-16}  \\
\ftt{+1.000000000100000e-15} &&\ftt{1.5e-09}&&\ftt{2.1e-16}    \\
\ftt{+1.000000000000000e+00} &&\ftt{0}&&\ftt{2.2e-16}\\
\\
&&$\berra =$ \ftt{8.2e-27} && $\berra =$ \ftt{4.7e-16} \\
&&$\berrb =$ \ftt{5.1e-02} &&$\berrb =$ \ftt{1.5e-06}\\
&&$\berrc =$ \ftt{5.1e-02} &&$\berrc =$ \ftt{6.7e-16}\\
\end{tabular}
\end{center}
\end{table}
\end{example}
\section{Min-max elementwise backward error}\label{sec_def_be}

Without loss of generality, we can assume the zero roots have been deflated and
that the resulting polynomial $p(z)$ has no root equal to zero so that $p_0\ne
0$. The
\emph{elementwise backward error for an approximate root} $\hz_k$ of $p(z)$ is
defined as
\begin{equation}\label{def.eberr}
        \etaea(\hz_k)
        = \min\{
        \eps: \ p(\hz_k)+\delp(\hz_k) =0, \
        |\delbp|\le \eps \boldsymbol\alpha\},
\end{equation}
where $\delp(z)=\sum_{i=0}^d \delp_i z^i$,
$\delbp=[\delp_0,\ldots,\delp_d]$,  the entries of
$\ba=[\alpha_0,\ldots,\alpha_d]$ are
nonnegative parameters, and the inequality  $|\delbp|\le \eps \boldsymbol\alpha$
holds componentwise. The parameters $\alpha_i$ allow freedom in how
perturbations are measured, for example, in an absolute sense with $\alpha_i= 1$ or
relative sense if $\alpha_i=|p_i|$.
It follows from $p(\hz_k)+\delp(\hz_k) =0$ that
\begin{equation}\label{pbound}
        |p(\hz_k)| = |\delp(\hz_k)|
        \le
        \sum_{i=0}^d |\delp_i| |\hz_k|^i
        \le
        \eps \sum_{i=0}^d \alpha_i |\hz_k|^i
\end{equation}
so that $\etaea(\hz_k) \ge |p(\hz_k)|/\sum_{i=0}^d \alpha_i |\hz_k|^i$.
It is easy to check that the lower bound is attained for the perturbations
$\delp_i = - (\sum_{i=0}^d \alpha_i |\hz_k|^i)^{-1} \alpha_i
\mathrm{sign}(\hz_k^i)p(\hz_k)$, $i=0,\ldots, d$.
Hence
\begin{equation}\label{eq.eberr}
        \etaea(\hz_k)
        = \frac{|p(\hz_k)|}{\sum_{i=0}^d \alpha_i |\hz_k|^i}.
\end{equation}
Note that the backward error \eqref{def.eberr}
is just a particular case of the componentwise backward error for an approximate
solution to a linear
system $Ax=b$ with rectangular matrix $A$ equal to the row vector
$\bp=[p_0,\ldots ,p_d]$ and $b=0$ \cite[Sec. 7.2]{high:ASNA2}.
So \eqref{eq.eberr} is a special case of Oettli and Prager's explicit expression
for the componentwise backward error for linear systems \cite{oepr64}.

Now with the particular choice $\ba=|\bp|=[|p_0|,\ldots, |p_d|]$ to measure the
perturbations, we obtain the
\emph{relative componentwise backward error}
\begin{equation}\label{eq.releberr}
        \eta^{\mathrm{elem}}_{|\bp|}(\hz_k)
        = \frac{|p(\hz_k)|}{\sum_{i=0}^d |p_i\hz_k^i|}.
\end{equation}

One of our contributions in this paper is
another choice for the vector of parameters $\ba$, weaker than
$\ba =|\bp|$ but still meaningful.
We rewrite the first upper bound in \eqref{pbound} as
\begin{equation}\label{eq.sum1}
\sum_{i=0}^d |\delp_i| |\hz_k|^i=
\sum_{i=0,p_i\ne 0}^d |\delp_i| |\hz_k|^i +
\sum_{i=0,p_i= 0}^d |\delp_i| |\hz_k|^i,
\end{equation}
and define
$$
j := \arg \max_i | p_i | |\hz_k|^i.
$$
Note that for such $j$, $p_j\ne 0$.
Now the term $|\delp_i||\hz_k|^i$ does not affect the order of
magnitude of the sum in \eqref{eq.sum1} when
\begin{equation}\label{def.betai}
        |\delp_i| |\hz_k|^i \le |\delp_j||\hz_k|^j
        \quad \Longleftrightarrow \quad
        \begin{cases}
        \displaystyle{\frac{1}{\beta_i}\frac{|\delp_i|}{ |p_i|}\le
        \frac{|\delp_j|}{|p_j|}}
        & \mbox{if $p_i\ne 0$,}\\
        \displaystyle{\frac{1}{\beta_i}|\delp_i| \le \frac{|\delp_j|}{|p_j|}}
        &\mbox{otherwise,}\
        \end{cases}
\end{equation}
with
\begin{equation}
\beta_i =  \\
\begin{cases}
\displaystyle{\frac{|p_j| |\hz_k|^j}{|p_i| |\hz_k|^i}} \geq 1 &\mbox{if $p_i\ne 0$,}\\
|p_j||\hz_k|^{j-i} & \mbox{otherwise}.
\end{cases}
\end{equation}
This suggests choosing $\ba = \btg$ to measure the perturbations
in \eqref{def.eberr}, where
$\btg=[\tilde\gamma_0,\ldots,\tilde\gamma_d]$ with
\begin{equation}\label{def.tildealpha}
        \tilde\gamma_i= \left\{\begin{array}{ll}
        \beta_i |p_i| &\mbox{if $p_i\ne 0$,}\\
        \beta_i    &\mbox{otherwise,}
  \end{array}\right.  \qquad i=0,\ldots,d.
\end{equation}
With this choice of parameters,
\begin{itemize}
\item larger perturbations are allowed on coefficients that do not participate
much to the upper bound in \eqref{pbound}, i.e.,  on the modulus of the residual
$p(\hz_k)$, but
\item the sparsity structure of the problem may not preserve, that is, if $p_i=0$
then $\delp_i\ne 0$ is allowed.
\end{itemize}
Note that the entries of $\btg$ depend on $\hz_k$.
Since
$$
\sum_{i=0}^d \widetilde\gamma_i |\hz_k|^i
= (d+1) |p_j\hz_k^j|= (d+1) \max_i |p_i| |\hz_k^i|,
$$
we find that
$$
\eta^{\mathrm{elem}}_{\btg}(\hz_k)
        = \frac{1}{d+1}\cdot\frac{|p(\hz_k)|}{\max_i |p_i| |\hz_k^i|}.
$$
Hence,
\begin{equation}\label{bndberr}
    \eta^{\mathrm{elem}}_{\btg}(\hz_k)\le
    \eta^{\mathrm{elem}}_{|\bp|}(\hz_k)\le
    (d+1)\eta^{\mathrm{elem}}_{\btg}(\hz_k).
\end{equation}

As mentioned in the introduction, we are interested in a global way of
measuring the backward error
for a given (usually computed) approximation $\bhz
= [\hz_1,\ldots,\hz_d]^T$ to all the zeros $\bz=[z_1,\ldots,z_d]^T$ of $p(z)$.
A natural extension of the definition for the backward error
for a single zero provided in \eqref{def.eberr} leads to
\begin{equation}\label{def.eberrmult}
        \eta^{\mathrm{elem}}_{\ba}(\bhz) := \min\{
        \eps: \ p(\bhz)+\delp_\mu(\bhz) =0,\
        |\delbp_\mu|\le \eps \ba\},
\end{equation}
where  $p(\bhz)$ denotes the vector $[p(\hz_1),\ldots,p(\hz_d)]^T$,
$$
        \delp_\mu(z)=\sum_{i=0}^d \delp_{\mu,i} z^i=p(z)-\mu\widetilde p(z),
        \quad\mu\in\C\setminus\{0\},
$$
with
\begin{equation}\label{def.pt}
\widetilde p(z) = p_\deg\prod_{j=1}^d (z-\hz_j)
=\widetilde{p}_\deg z^\deg+\widetilde p_{\deg-1} z^{\deg-1}+\cdots+
\widetilde{p}_1z+\widetilde{p}_0
\end{equation}
fixed, and $\delbp_\mu=[\delp_{\mu,0},\ldots,\delp_{\mu,\deg}]$ with
$\delp_{\mu,i}=p_i-\mu\widetilde{p}_i$.
Then
\begin{equation}\label{eq.elberr}
\eta^{\mathrm{elem}}_{\ba}(\bhz) =
\begin{cases}
\displaystyle{\min_{\mu\in\C\setminus\{0\}}
\ \max_{i, \alpha_i \neq 0} \frac{| \delp_{\mu,i} |}{\alpha_i }}
&\mbox{if $\delp_{\mu,i}=0$ when $\alpha_i=0$},\\
\infty &\mbox{otherwise}.
\end{cases}
\end{equation}
When $p(z)$ has real coefficients and the approximate roots $\bhz$ are symmetric
with respect to the real axis, then the coefficients of $\tp(z)$ are real and we
can we can minimize over $\mu\in\R\setminus\{0\}$ in \eqref{eq.elberr}.
Then in this case, \eqref{eq.elberr} can be rewritten as a linear programming
problem that can be solved by the simplex method.

When $\ba=|\bp|$ the elementwise relative backward error in
\eqref{eq.releberr} is just an upper bound on \eqref{eq.elberr} corresponding to
choosing $\mu=1$.

Let us define
\begin{equation}\label{def.gammai}
\gamma_i = \min_{|z| \geq 0} \beta_i(z), \qquad
\beta_i(z) = \left\{
  \begin{array}{ll}
  \displaystyle{\max_{0\le j\le d}
  \frac{|p_jz^j|}{|p_iz^i|}} &\mbox{if $p_i \neq 0$,}\\
  \displaystyle{\max_{0\le j\le d}
  \frac{|p_jz^j|}{|z|^i}} & \mbox{if $p_i = 0$},
  \end{array}\right.
\qquad i=0,\ldots,d,
\end{equation}
and consider the generalization to $\btg$ in
\eqref{def.tildealpha},
\begin{equation}\label{def.gammat}
        \widetilde\gamma_i= \left\{\begin{array}{ll}
         \gamma_i |p_i| &\mbox{if $p_i\ne 0$,}\\
         \gamma_i  &\mbox{if $p_i= 0$,}
  \end{array}\right.
  \qquad i=0,\ldots,d.
\end{equation}
\begin{definition}[Min-max elementwise backward error]\label{def01}
The min-max elementwise backward error for
the approximate zeros
$\bhz=[\hz_1,\ldots,\hz_d]^T$ of $p(z) = \sum_{i=0}^\deg p_i z^i$
is
$\eta^{\mathrm{elem}}_{\btg}(\bhz)$ in~\eqref{def.eberrmult}
with $\btg=[\widetilde\gamma_0,\ldots,\widetilde\gamma_\deg]$ defined  in~\eqref{def.gammat}.
\end{definition}
The naming \emph{min-max} comes from the min-max characterization of the
$\gamma_i$ in \eqref{def.gammai}.
Note that the latter
do not depend on the approximate zeros $\hz_k$, $k=1,\ldots,d$.
Because $\gamma_i\ge 1$ if $p_i\ne 0$, our choice for $\widetilde\gamma_i$  allows a
larger relative error on those coefficients having a corresponding $\gamma_i$
larger than one. For $p_i=0$, we allow
a certain absolute error without changing the order of magnitude of the upper
bound (\ref{pbound}) for any of the zeros $\hz_k$.
This does not preserve
sparsity in the data but unlike for the elementwise relative backward error, the
min-max elementwise backward error is always finite. Indeed, we will show in the
next section that for $p_i=0$, $\gamma_i>0$ so $\widetilde\gamma_i\ne 0$.

\section{Connection between the $\gamma_i$ and the tropical roots}\label{sec_compute_gamma}


Based on the polynomial $p(z)$, consider the corresponding tropical polynomial $\tpol (x)$ in the max-times
semiring $\Rmaxt$ consisting of the set of nonnegative real numbers $\Rp$
with the operations $\tplus$ and $\tmult$. The $\tplus$ operation is taking the maximum value of the terms and the $\tmult$
operation is the classical multiplication.
The tropical polynomial $\tpol (x)$ based on $p(z)$ is defined as
\begin{equation}\label{def.tp}
\tpol : \Rp \rightarrow \Rp: x \mapsto
\bigoplus_{i=0}^\deg |p_i| x^i = \max_{0\le i\le \deg} |p_i| x^i .
\end{equation}
If $p_0 = p_1 = \cdots = p_{m-1} = 0$, then
zero is a tropical root of $\tpol (x)$ with multiplicity $m_0=m$ ($m_0=0$ is
$p_0\ne 0$).
The (nonzero) tropical roots are points $x$ in $\Rp$ at which the maximum in
\eqref{def.tp} is attained for at least two values of $i$ for
this specific $x$-values \cite{abg14}.
Since $p_\deg \neq 0$, there are $t\le d$ distinct tropical roots
$$
0<\tr_1<\cdots<\tr_t
$$
with $\tr_\ell$ of multiplicity $m_\ell$,
and $\sum_{\ell=0}^t m_\ell = \deg$.
These tropical roots can be computed from
the Newton polygon associated with $p(z)$,
i.e., the upper boundary of the convex hull of the set of points
$(j,\log|p_j|)$, $j = 0,1,\ldots,\deg$ resulting in the points
$(k_\ell,\log|p_{k_\ell}|)$, $\ell=1,\ldots,t$
with
$$
        k_0 = 0 < k_1 <\cdots< k_{t-1} < k_t = \deg.
$$
The opposites of the slopes of the segments
of this upper boundary are the logarithm of the tropical roots. Hence, if
$(k_{\ell-1}, \log|p_{k_{\ell-1}}|)$ and $(k_{\ell}, \log|p_{k_{\ell}}|)$ are the two
endpoints of such a segment, it follows that
\begin{equation}\label{def.troproot}
\tr_\ell = \bigg(\frac{|p_{k_{\ell-1}}|}{|p_{k_\ell}|}\bigg)^{1/m_\ell},
\quad m_\ell = k_\ell - k_{\ell-1}, \quad \ell=1,2,\ldots,t.
\end{equation}
These tropical roots can be computed in $\OO(\deg)$ operations
\cite[Prop.~1]{r957}. For more details on tropical roots and corresponding
applications in polynomial eigenvalue problems, we refer the interested reader
to \cite{r958,r957,hmt13,r953,d665,VBT2017}.

When computing the zeros of a polynomial $p(z)$, we assume that $p_0
\neq 0$, i.e., the possible roots at zero are already divided out exactly.
Hence, all tropical roots $\tr_l$ will be nonzero.

\begin{theorem}\label{thm.gamma}
Given a polynomial $p(z)=\sum_{i=0}^d p_i z^i$
and its associated tropical
polynomial $\tpol(x)=\max_{0\le i\le d} |p_i| x^i
=\max_{0\le i\le t} |p_{k_i}| x^{k_i}$,
with $0=k_0 < k_1 <\cdots< k_t = \deg$ and
tropical roots $\tr_\ell$, $\ell =1,2,\ldots,t$
as in \eqref{def.troproot},
we have that for each $i\in\{1,\ldots,\deg\}$ and corresponding $\ell$
such that $k_{\ell-1} \leq i \leq k_\ell$,
the parameters $\gamma_i$ in \eqref{def.gammai} are
given by
$$
\gamma_i =
\begin{cases}
\displaystyle{\frac{\tr_\ell^{k_{\ell-1}}|p_{k_{\ell-1}}|}{\tr_\ell^i |p_i|}}
= \frac{\tr_\ell^{k_{\ell}}|p_{k_{\ell}}|}{\tr_\ell^i |p_i|}
& \mbox{if $p_i\ne 0$},\\ \mystrut{.8cm}
\displaystyle{\frac{\tr_\ell^{k_{\ell-1}}|p_{k_{\ell-1}}|}{\tr_\ell^i} =
\frac{\tr_\ell^{k_{\ell}}|p_{k_{\ell}}|}{\tr_\ell^i}}
&\mbox{if $p_i= 0$}.
\end{cases}
$$
\end{theorem}
\begin{proof}
If $p_i\ne 0$, then by \eqref{def.gammai},
$\gamma_i = \min_{|z| > 0} \beta_i(z)$ with
$\beta_i(z) =\max_j |p_j z^j|/|p_i z^i|$.
Hence,
$$
\beta_i(\tr_\ell) = \frac{\max_j |p_j \tr_\ell^j|}{|p_i \tr_\ell^i|} = \frac{ | p_{k_{\ell-1}} | \tr_\ell^{k_{\ell-1}}}{|p_i| \tr_\ell^i}
= \frac{| p_{k_{\ell}} | \tr_\ell^{k_{\ell}}}{|p_i| \tr_\ell^i}.
$$
Suppose that $\tr_{\ell'-1} \leq |z| \leq \tr_{\ell'} \leq \tr_\ell$, then we get
\begin{eqnarray*}
\beta_i(z) &=& \frac{\max_j |p_j z^j|}{|p_i z^i|}
= \frac{|p_{k_{\ell'-1}} z^{k_{\ell'-1}} |}{|p_i z^i|}
\geq \frac{|p_{k_{\ell-1}} z^{k_{\ell-1}} |}{|p_i z^i|}
= \frac{|p_{k_{\ell-1}}|}{|p_i|} \frac{1}{|z|^{i-k_{\ell-1}}} \\
&\geq& \frac{|p_{k_{\ell-1}}|}{|p_i|} \frac{1}{\tr_\ell^{i-k_{\ell-1}}}
= \beta_i(\tr_\ell) .
\end{eqnarray*}
A similar argument can be followed when
$\tr_\ell \leq \tr_{\ell'} \leq |z| \leq \tr_{\ell'+1}$ to prove that also in this case
$\beta_i(z) \geq \beta_i(\tr_\ell)$.
Hence, the minimum value for $\beta_i(z)$ is for $z = \tr_\ell$.
This gives us the expression for $\gamma_i$.
The expression when $p_i = 0$ is proved in a similar way.
\end{proof}
When $p_0\ne 0$, there are no zero tropical roots and no indices $k_\ell$ from
the Newton polygon such that $p_{k_\ell}=0$ so $\gamma_i>0$ for all $i$.
%
Instead of computing the min-max elementwise backward error
$$
\berrc(\bhz)
=\min_{\mu\in\C\setminus\{0\}}
\ \max_{0\le i\le \deg} \frac{|p_i-\mu\tp_i|}{\widetilde\gamma_i},
$$
it is easier to compute the upper bound $\berrc$ given by
\begin{equation}\label{ineq.upperbnd}
  \berrc= \max_{0<i\le \deg} \frac{|p_i-\tp_i|}{\widetilde\gamma_i}
\end{equation}
for which the coefficients $\tp_i$ are obtained by constructing
$\tp(z)$ in \eqref{def.pt} using multiple precision.

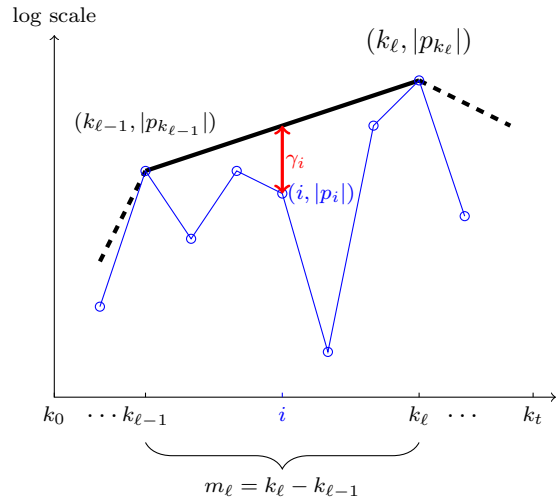
\begin{figure}
\centering
\begin{tikzpicture}[scale=0.6]

\draw[<->] (11,1) -- (0,1) -- (0,9) node[above]{\footnotesize$\log$ scale};

\draw[ultra thick](2,6) node[above,yshift = 0.3cm]
{\footnotesize$(k_{\ell-1},|p_{k_{\ell-1}}|)$} --
(8,8) node[above,yshift = 0.2cm]{$(k_{\ell},|p_{k_{\ell}}|)$};
\draw[dashed,ultra thick] (1,4) -- (2,6);
\draw[dashed,ultra thick] (8,8) -- (10,7);

\draw (2,1.1) -- (2,1) node[below]{\footnotesize$k_{\ell-1}$};
\draw (8,1.1) -- (8,1) node[below]{\footnotesize$k_{\ell}$};
\draw (0,1.1) -- (0,1) node[below]{\footnotesize$k_0$};
\draw (10.5,1.1) -- (10.5,1) node[below]{\footnotesize$k_t$};
\draw [blue] (5,1.1) -- (5,1) node[below]{\footnotesize$i$};
\node at (1.1,0.6) {\ldots};
\node at (9,0.6) {\ldots};
\draw [decorate,decoration={brace,amplitude=10pt,mirror}]
(2,0) -- (8,0) node[black,midway,yshift=-0.6cm]
{\footnotesize $m_\ell = k_\ell-k_{\ell-1}$};

\draw[blue] (1,3) -- (2,6) -- (3,4.5) -- (4,6) -- (5,5.5) -- (6,2) -- (7,7) -- (8,8) -- (9,5);
\draw [blue] (1,3) circle [radius=0.1];
\draw [blue] (2,6) circle [radius=0.1];
\draw [blue] (3,4.5) circle [radius=0.1];
\draw [blue] (4,6) circle [radius=0.1];
\draw [blue] (5,5.5) circle [radius=0.1] node[right]{\!\footnotesize$(i,|p_i|)$};
\draw [blue] (6,2) circle [radius=0.1];
\draw [blue] (7,7) circle [radius=0.1];
\draw [blue] (8,8) circle [radius=0.1];
\draw [blue] (9,5) circle [radius=0.1];

\draw [<->, very thick, red] (5,5.5) -- (5,7);
\node [red] at (5.3,6.2) {\footnotesize$\gamma_i$};

\end{tikzpicture}
\caption{The factor $\gamma_i$ as fraction of a point on the convex hull and the modulus of the corresponding polynomial
coefficient $|p_i|$.\label{fig951}}
\end{figure}

In Figure~\ref{fig951} the parameter $\gamma_i$ is graphically indicated as the
fraction of the value of the convex hull and the modulus of the corresponding
polynomial coefficient $|p_i|$. Note that $\gamma_{k_\ell}=1$,
$\ell=0,\ldots,t$, which in the log scale on Figure~\ref{fig951} corresponds to
zero.
The min-max backward error is equal to $\epsilon$ when each absolute error
$|\Delta p_i|$ is $\eps$ times smaller than the convex hull and when there is
one or more of these absolute errors just $\eps$ times smaller. In
Figure~\ref{figaaa} (left) this is shown when the value $|p_i|$ is between the
corresponding point on the convex hull and $\eps$ times smaller.
Here $\Delta p_i$ is very small so $p_i\approx \widetilde p_i$.
Figure~\ref{figaaa} (right) illustrates the case when $|p_{i'}|$ is less than
$\eps$ times the corresponding point on the convex hull. Here
$p'_i$ is very small so $\Delta p'_i=\widetilde p'_i$.

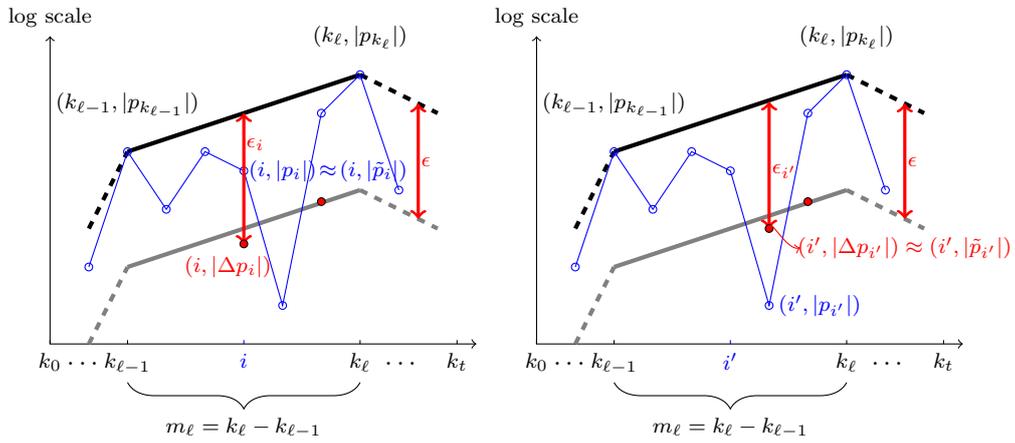
\begin{figure}
\centering
\begin{tikzpicture}[scale=0.51]

\draw[<->] (11,1) -- (0,1) -- (0,9) node[above]{\footnotesize$\log$ scale};

\draw[ultra thick](2,6) node[above,yshift = 0.3cm]
{\footnotesize$(k_{\ell-1},|p_{k_{\ell-1}}|)$} --
(8,8) node[above,yshift = 0.2cm]{\footnotesize$(k_{\ell},|p_{k_{\ell}}|)$};
\draw[dashed,ultra thick] (1,4) -- (2,6);
\draw[dashed,ultra thick] (8,8) -- (10,7);

\draw (2,1.1) -- (2,1) node[below]{\footnotesize$k_{\ell-1}$};
\draw (8,1.1) -- (8,1) node[below]{\footnotesize$k_{\ell}$};
\draw (0,1.1) -- (0,1) node[below]{\footnotesize$k_0$};
\draw (10.5,1.1) -- (10.5,1) node[below]{\footnotesize$k_t$};
\draw [blue] (5,1.1) -- (5,1) node[below]{\footnotesize $i$};
\node at (.9,0.5) {\ldots};
\node at (9.1,0.5) {\ldots};
\draw [decorate,decoration={brace,amplitude=10pt,mirror}]
(2,0) -- (8,0) node[black,midway,yshift=-0.6cm]
{\footnotesize $m_\ell = k_\ell-k_{\ell-1}$};

\draw[gray,ultra thick](2,3) -- (8,5);
\draw[gray,dashed,ultra thick] (1,1) -- (2,3);
\draw[gray,dashed,ultra thick] (8,5) -- (10,4);

\draw[blue] (1,3) -- (2,6) -- (3,4.5) -- (4,6) -- (5,5.5) -- (6,2) -- (7,7) -- (8,8) -- (9,5);
\draw [blue] (1,3) circle [radius=0.1];
\draw [blue] (2,6) circle [radius=0.1];
\draw [blue] (3,4.5) circle [radius=0.1];
\draw [blue] (4,6) circle [radius=0.1];
\draw [blue] (5,5.5) circle [radius=0.1]node[right]
{\!\footnotesize$(i,|p_i|)\! \approx\! (i,|\tilde{p}_i|)$};
\draw [blue] (6,2) circle [radius=0.1];
\draw [blue] (7,7) circle [radius=0.1];
\draw [blue] (8,8) circle [radius=0.1];
\draw [blue] (9,5) circle [radius=0.1];

\draw [<->, very thick, red] (9.5,7.25) -- (9.5,4.25);
\node [red] at (9.7,5.75) {\footnotesize$\epsilon$};
\draw [<->, very thick, red] (5,3.6) -- (5,7);
\node [red] at (5.3,6.2) {\footnotesize$\epsilon_i$};

\draw [fill=red] (5,3.6) circle [radius=0.1];
\node [red,yshift=-0.3cm,xshift=0.3cm] at (4,3.6) {\footnotesize$(i,|\Delta p_i|)$};

\draw [fill=red] (7,4.7) circle [radius=0.1];
\end{tikzpicture}
\begin{tikzpicture}[scale=0.51]

\draw[<->] (11,1) -- (0,1) -- (0,9) node[above]{\footnotesize$\log$ scale};

\draw[ultra thick](2,6) node[above,yshift = 0.3cm]
{\footnotesize$(k_{\ell-1},|p_{k_{\ell-1}}|)$} -- (8,8)
node[above,yshift = 0.2cm]{\footnotesize$(k_{\ell},|p_{k_{\ell}}|)$};
\draw[dashed,ultra thick] (1,4) -- (2,6);
\draw[dashed,ultra thick] (8,8) -- (10,7);

\draw (2,1.1) -- (2,1) node[below]{\footnotesize$k_{\ell-1}$};
\draw (8,1.1) -- (8,1) node[below]{\footnotesize$k_{\ell}$};
\draw (0,1.1) -- (0,1) node[below]{\footnotesize$k_0$};
\draw (10.5,1.1) -- (10.5,1) node[below]{\footnotesize$k_t$};
\draw [blue] (5,1.1) -- (5,1) node[below]{\footnotesize$i'$};
\node at (.9,0.5) {\ldots};
\node at (9.1,0.5) {\ldots};
\draw [decorate,decoration={brace,amplitude=10pt,mirror}]
(2,0) -- (8,0) node[black,midway,yshift=-0.6cm]
{\footnotesize $m_\ell = k_\ell-k_{\ell-1}$};

\draw[gray,ultra thick](2,3) -- (8,5);
\draw[gray,dashed,ultra thick] (1,1) -- (2,3);
\draw[gray,dashed,ultra thick] (8,5) -- (10,4);

\draw[blue] (1,3) -- (2,6) -- (3,4.5) -- (4,6) -- (5,5.5) -- (6,2) -- (7,7) -- (8,8) -- (9,5);
\draw [blue] (1,3) circle [radius=0.1];
\draw [blue] (2,6) circle [radius=0.1];
\draw [blue] (3,4.5) circle [radius=0.1];
\draw [blue] (4,6) circle [radius=0.1];
\draw [blue] (5,5.5) circle [radius=0.1];
\draw [blue] (6,2) circle [radius=0.1] node[right]{\footnotesize$(i',|p_{i'}|)$};
\draw [blue] (7,7) circle [radius=0.1];
\draw [blue] (8,8) circle [radius=0.1];
\draw [blue] (9,5) circle [radius=0.1];

\draw [<->, very thick, red] (9.5,7.25) -- (9.5,4.25);
\node [red] at (9.7,5.75) {\footnotesize$\epsilon$};
\draw [<->, very thick, red] (6,4.0) -- (6,7.3);
\node [red] at (6.4,5.5) {\footnotesize$\epsilon_{i'}$};

\draw [fill=red] (6,4.0) circle [radius=0.1];
\node [red] at (9.5,3.5) {\footnotesize$(i',|\Delta p_{i'}|) \approx (i', |\tilde{p}_{i'}|)$};

\draw [red,->] (6,4.0) to [out=90,in=195](6.8,3.5);
\draw [fill=red] (7,4.7) circle [radius=0.1];
\end{tikzpicture}
\caption{The min-max backward error is $\eps$.
On the left, the value of $|p_i|$ is within the band between the convex hull and
$\eps$ times the convex hull.
On the right, the value of $|p_{i'}|$ is less than $\eps$ times the convex hull
\label{figaaa}}
\end{figure}



\begin{example}[Mastronardi and Van Dooren's example~\cite{MVD2015}]
Let $u$ denote the unit roundoff and consider the polynomial
$$
p(z) = z^2 -2\beta z -1
$$
with zeros $z_{1,2} = \beta \pm \sqrt{\beta^2+1}$, where $\beta = 2^{-t} +
2^{-2t}$ with $2^{-2t} \leq u/2$ and $2^{-t} \approx \sqrt{u}$.
These zeros are well conditioned.
The approximations
$\hz_{1,2} = 2^{-t}\pm 1$ to the zeros $z_{1,2}$ both have a
relative error of order $u$.
These approximations
are the exact zeros of the polynomial
$$
\tp(z) = z^2 - 2^{-t+1} z +2^{-2t}-1 .
$$
For the polynomial $p(z)$, it is easy to check that
$$
\gamma_0 = \gamma_2 = 1 \quad\mbox{and}\quad \gamma_1 = (2\beta)^{-1} = \OO(u^{-1/2}) \gg 1.
$$
This results in the backward errors
$$
\berrb(\bhz) = \min_{\mu\in\R}\max_{0\le i\le 2}
\frac{|\tp_i - p_i|}{|p_i|} \approx u^{1/2},
\quad
\berrc(\bhz) = \min_{\mu\in\R} \max_{0\le i\le 2}\frac{|\tp_i - p_i|}{|p_i| \gamma_i}
= 2^{-2t+1} \leq u.
$$
Other approximations $\hz_1$ and $\hz_2$ of the zeros with
a relative error of the size of the unit roundoff lead to similar results.

\end{example}

\section{A root solver based on companion linearization}\label{sec_algo}

In this section we describe an algorithm for computing all the zeros of a
polynomial $p(z) = \sum_{i=0}^\deg p_i z^i$. It is based on a companion
linearization of the grade $d+1$ polynomial $0\cdot z^{d+1}+p(z)$,
a suitable scaling/balancing of the
linearization, and the use of a QZ algorithm with an appropriate deflation
strategy for the eigenvalues at infinity.
The min-max elementwise backward error of the resulting algorithm is studied in
Section~\ref{sec_be}.


We transform the original problem, i.e., $p(z) =0$  into the generalized eigenvalue problem
$$
\begin{bmatrix}
p_{\deg} & p_{\deg-1} & \cdots & p_1 & p_0 \\
1 & -z & && \\
  & 1 & -z & & \\
  & & \ddots & \ddots & \\
  & & & 1 & -z
\end{bmatrix}
\begin{bmatrix}
z^{\deg} \\
z^{\deg-1} \\
\vdots \\
z \\
1
\end{bmatrix}
=
\begin{bmatrix}
p(z) \\
0 \\
\vdots \\
0 \\
0
\end{bmatrix}.
$$
The pencil on the left of the equality, which we write $C(z)=A-z B$ with
\begin{equation}\label{def.compa}
A = \begin{bmatrix}
p_{\deg} & p_{\deg-1} & \cdots & p_1 & p_0 \\
1 &  & && \\
  & 1 &  & & \\
  & & \ddots &  & \\
  & & & 1 &
\end{bmatrix},
\qquad
B =
\begin{bmatrix}
0 &  &  &  &  \\
 & 1 & && \\
  & &  1& & \\
  & &  & \ddots & \\
  & & &  & 1
\end{bmatrix}
\end{equation}
is the  $(\deg+1)\times (\deg+1)$ companion linearization of $0\cdot z^{d+1}+p(z)$. It has an eigenvalue at
infinity and its finite eigenvalues are the roots of $p(z)$.
We then apply a two-sided diagonal scaling to the pencil $C(z)$,
$$\hC(z) = D_l C(z) D_r=\hA - z \hB$$
with diagonal matrices $D_l,D_r$ constructed such that
$\hA$ is balanced in the sense that its nonzero entries are in modulus bounded
by $1$ and the diagonal of $\hB$ is graded.
This is done as follows. We define
\begin{equation}\label{def.ttr}
        \ttr_1,\ttr_2,\ldots,\ttr_d :=
        \underbrace{\tr_1,\ldots,\tr_1}_{\mbox{\footnotesize$m_1$ times}},
        \underbrace{\tr_2,\ldots,\tr_2}_{\mbox{\footnotesize$m_2$ times}},
        \ldots,
        \underbrace{\tr_t,\ldots,\tr_t}_{\mbox{\footnotesize$m_t$ times}}
\end{equation}
with $\tr_\ell$ as in \eqref{def.troproot} and construct
\begin{equation}\label{def.Dl}
D_l = \diag\biggl(\frac{1}{|p_\deg|}, 1, \prod_{j=\deg}^\deg \ttr_j,
\prod_{j=\deg-1}^\deg \ttr_j, \ldots, \prod_{j=2}^\deg \ttr_j\biggr)
\end{equation}
and
\begin{equation}\label{def.Dr}
D_r = \diag\biggl(1, 1/\prod_{j=\deg}^\deg \ttr_j,
1/\prod_{j=\deg-1}^\deg \ttr_j, \ldots, 1/\prod_{j=1}^\deg \ttr_j\biggr).
\end{equation}
The matrices of the scaled pencil $\hC(z) = \hA - z \hB$ have the form
$$
\hA =
\begin{bmatrix}
\ha_{\deg} & \ha_{\deg-1} & \cdots & \ha_1 & \ha_0 \\
1 &  & && \\
  & 1 &  & & \\
  & & \ddots &  & \\
  & & & 1 &
\end{bmatrix},
\qquad
\hB =
\begin{bmatrix}
0 &  &  &  &  \\
 & \hb_{1} & && \\
  & &  \hb_{2} & & \\
  & &  & \ddots & \\
  & & &  & \hb_{\deg}
\end{bmatrix}
$$
with
$$
|\ha_i| =
\begin{cases}
\gamma_i^{-1} & \mbox{if $ p_i \neq 0$}, \\
0 &\mbox{if $p_i = 0$},
\end{cases}
\qquad\qquad
| \hb_i | = \ttr_{\deg-i+1}^{-1},
$$
and $\gamma_i$ as in Theorem~\ref{thm.gamma}.
Now the parameters $\gamma_{k_\ell}$, $\ell = 1,2,\ldots,t$ corresponding
to the indices of the Newton polytope are equal to one. The other $\gamma_i$ values
are larger than one, i.e., $\gamma_i^{-1}$ is smaller than one for these other
values. Hence, all elements of the top row of the transformed matrix $\hA$ are
in modulus less than or equal to one. The nonzero diagonal part of $\hB$
contains the inverses of the tropical roots and since $\ttr_j\le \ttr_{j+1}$,
we have that $\hb_j\le \hb_{j+1}$.
Note that once the tropical roots are
computed, the factors $\gamma_i$ naturally appear as inverses of the modulus of
the corresponding elements $\ha_i$ when $p_i \neq 0$.

The trivial eigenvalue at infinity is easily deflated:
let $G$ be a $2\times 2$ Givens rotation such that
$G\twobyone{\ha_\deg}{1} =\twobyone{\check{a}_\deg}{0}$
and embed it in $I_{\deg+1}$ as the $2\times 2$ leading block to form
$\widetilde G$.
Then
\begin{equation}\label{def.wAwB}
\wA=\widetilde{G}\hA =
\begin{bmatrix}
\breve{a}_{\deg} & \breve{a}_{\deg-1} & \cdots & \breve{a}_1 & \breve{a}_0 \\
0 & \widetilde{a}_{\deg-1} & \cdots & \widetilde{a}_1 & \widetilde{a}_0 \\
  & 1 &  & & \\
  & & \ddots &  & \\
  & & & 1 &
\end{bmatrix},
\qquad
\wB=\widetilde{G}\hB =
\begin{bmatrix}
0 &\breve{b}_1  &  &  &  \\
 & \widetilde b_{1} & && \\
  & &  \hb_{2} & & \\
  & &  & \ddots & \\
  & & &  & \hb_{\deg}
\end{bmatrix}.
\end{equation}
The roots of $p(z)$ are the $d$ (finite) eigenvalues of the trailing
$\deg\times\deg$ submatrix of $\wA- z\wB$.
The $\deg\times \deg$ trailing submatrix of $\wA$
remains well-balanced with entries in the first row
that are of modulus less or equal to $\sqrt{2}/2$.
The diagonal entries of $\widetilde B$ remained ordered with increasing moduli.

Large differences in the magnitude of the entries of $\wB$ generically lead to
a large difference in the magnitude of the eigenvalues of $\wA-z \wB$.
As already observed in \cite[Section 3]{VBT2017}, when the difference in the
entries of $\wB$ is larger than $\emach^{-1}$, $\emach$ being the machine precision,
the LAPACK implementation of the QZ algorithm may decide too quickly to deflate
an eigenvalue and declare it to be at infinity.
So we slightly modify the LAPACK routines \texttt{xHGEQZ} such that besides
the trivial eigenvalues at infinity only finite eigenvalues
are generated. Note that the latter can be very large when they correspond to
exact infinite eigenvalues.
To be more specific we replace the value of \texttt{BTOL} by the smallest
positive nonzero floating point number in strategic places in the fortran code
as to avoid that a specific entry of $\wB$ is explicitly set to zero,
thereby leading to a computed infinite eigenvalue.
At the same time, we also increase the maximum number of iterations MAXIT.
The steps of our approach are summarized in Algorithm~\ref{alg01}.
\begin{algorithm}
\caption{Computes all zeros of a polynomial $p(z) = \sum_{i=0}^\deg p_i z^i$\label{alg01}}
\begin{algorithmic}
     \State{Input: the coefficients $p_i$, $i=0,1,\ldots,\deg$ of $p(z)$}.
     \State{Output: the $d$ zeros of $p(z)$}.
     \State{}
     \State{Construct the companion pencil $C(z) = A - z B$ of $p(z)$ as in
     \eqref{def.compa}}.
     \State{Compute the tropical roots $\ttr_j$, $j=1,2,\ldots,\deg$ as in
     \eqref{def.ttr}}.
     \State{Scale/balance the companion pencil: $\hC(z) = D_l C(z) D_r$
       with $D_l,D_r$ as in \eqref{def.Dl}--\eqref{def.Dr}}.
     \State{Deflate the eigenvalue at infinity with
     a Givens rotation as in \eqref{def.wAwB}}.
      \State{Compute the eigenvalues of the deflated pencil using
      the QZ algorithm implemented with a strict deflation criterion for the
      detection of eigenvalues at infinity.}
      \State{Return these eigenvalues as roots of $p(z)$.}
\end{algorithmic}
\end{algorithm}

In Section~\ref{sec_be}, we show that under certain assumptions on the graded
character of the backward error for the generalized Schur form obtained after
applying a QZ algorithm with strict deflation criterion for eigenvalues
at infinity,
Algorithm~\ref{alg01} is min-max elementwise backward stable according to
Definition \ref{def01}. In Section~\ref{sec_num_exp} several numerical
experiments will be given illustrating the backward stable behaviour of the
newly designed algorithm.

\section{Min-max backward error for Algorithm~\ref{alg01}}\label{sec_be}
In \cite[Section~6]{VBT2017} we gave numerical evidence for the following
assumption that is required for our backward error analysis of Algorithm~\ref{alg01}.
\begin{assumption}\label{ass01}
The QZ algorithm with a strict deflation criterion for eigenvalues at infinity
applied to $\widetilde{G}(\hA-z\hB)$ in \eqref{def.wAwB} computes the exact generalized
Schur form of the matrix pencil
$$
(\hA + \dhA) - z (\hB + \dhB),
$$
where all entries in $\dhA$ have  modulus of size
$\OO(\emach)$, $\emach$ being the machine precision,
and the entries in column $i$ of $\dhB$ have modulus of
size
$\OO(\ttr_{\deg-i+2}^{-1}\emach)$, with $\ttr_i$ as in \eqref{def.ttr},
except for the first column which is equal to zero.
\end{assumption}
It follows then that under  Assumption~\ref{ass01} the backward error on the matrix
$\hB$ has a graded structure if $\hB$ is graded since $\ttr_1\le \ttr_2\le\cdots\le\ttr_d$.

\begin{theorem}
Algorithm~\ref{alg01} applied to $p(z)=\sum_{i=0}^\deg p_iz^i$ is min-max elementwise backward stable
under Assumption~\ref{ass01}, that is, it computes roots
$\bhz=[\hz_1,\ldots,\hz_d]^T$ of $p(z)$ with min-max elementwise backward error
$\berrc(\bhz)=\OO(\emach)$.
\end{theorem}
\begin{proof}
To prove this theorem, we transform the matrix pencil $(\hA + \dhA) - z (\hB +
\dhB)$ into
\begin{equation}\label{eq20}
P \left[ (\hA + \dhA) - z (\hB + \dhB) \right] Q = (\hA + \dhA') - z \hB
\end{equation}
with $P$ and $Q$ nonsingular such that the resulting error is fully concentrated on the first row of $\hA$,
i.e., $\dhA'$ is zero except possibly for its first row.
The absolute error on each of the elements $\ha_i$ in the first row is of order $\OO(\emach)$.
We then show that performing the inverse of the original scaling/balancing
operation, i.e., $D_l^{-1} \dhA' D_r^{-1}$ with $D_l,D_r$ as in
\eqref{def.Dl}--\eqref{def.Dr} leads to a min-max backward error
of size $\OO(\emach)$.

Let us first concentrate on moving all errors towards the first row of $\hA$.
This can be done using several steps as in a Gaussian-elimination  algorithm.
In each of these steps, the error $\dhA$ stays of the order $\OO(\emach)$
while the error in $\dhB$ maintains the graded structure. Also the introduced zeros are maintained in the subsequent steps.
To indicate the order in which the elements are restored in their structured form, we use the same notation as in
\cite[Section 4]{n643} for a $4 \times 4$ example.
$$
\left[
\begin{array}{cccc}
\ha_3 + \dha_3' & \ha_2 + \dha_2' & \ha_1 + \dha_1' & \ha_0 + \dha_0' \\
1^{(8)} & 0^{(6)} & 0^{(4)} & 0^{(2)} \\
 0^{(9)} & 1^{(6)} & 0^{(4)} & 0^{(2)} \\
 0^{(9)} & 0^{(11)} & 1^{(4)} & 0^{(2)}
\end{array}
\right]
-z
\left[
\begin{array}{cccc}
0 & 0^{1)} & 0^{(1)} & 0^{(1)}  \\
0 & \hb_{1}^{(7)} & 0^{(5)} & 0^{(3)} \\
0  & 0^{(10)} &  \hb_{2}^{(5)} & 0^{(3)} \\
0  & 0^{(10)} & 0^{(12)} & \hb_{3}^{(3)}
\end{array}
\right] .
$$
Each of the absolute errors $\dha_i'$ is of the size $\OO(\emach)$. Reversing
the scaling/balancing operation leads to an absolute error $\delp_i$ on the
$i$th initial coefficient $p_i$ of $p(z)$ of size
\begin{eqnarray}\label{eq.dpi}
|\delp_i| &=& |p_d| |\dha_i'| \prod_{j=i+1}^\deg \ttr_j \\
&\stackrel{(\ref{def.ttr})}{=}&|p_d| |\dha_i'|  \underbrace{\tr_\ell,\ldots,\tr_\ell}_{\mbox{\footnotesize$(k_\ell-i)$ times}},
        \underbrace{\tr_{\ell+1},\ldots,\tr_{\ell+1}}_{\mbox{\footnotesize$m_{\ell+1}$ times}},
        \ldots,
        \underbrace{\tr_t,\ldots,\tr_t}_{\mbox{\footnotesize$m_t$ times}} \\
&\stackrel{(\ref{def.troproot})}{=}
& |p_d| |\dha_i'| \tr_\ell^{k_\ell-i} \frac{|p_{k_\ell}|}{|p_{k_{\ell+1}}|} \frac{|p_{k_{\ell+1}}|}{|p_{k_{\ell+2}}|} \cdots \frac{|p_{k_{t-1}}|}{|p_{k_{t}}|} \\
&=& |p_d| |\dha_i'| \tr_\ell^{k_\ell-i} \frac{|p_{k_\ell}|}{|p_d|}
=  |\dha_i'| \tr_\ell^{k_\ell-i} |p_{k_\ell}|
\end{eqnarray}
with $k_{\ell-1} \leq i \leq k_{\ell}$.
On using \eqref{eq.elberr}, \eqref{def.gammat}, Theorem~\ref{thm.gamma}, and \eqref{eq.dpi} we have
that
$$
        \berrc(\bz) \le \max_{i}
        \frac{|\delp_i|}{\widetilde\gamma_i}
        =\max_i{|\dha_i'|} = \OO(\emach).
        \myendproof
$$
\noqed
\end{proof}

\section{Numerical experiments}\label{sec_num_exp}
In section~\ref{secSCALAR}, the \MATLAB\  \texttt{roots} function is
compared to our new algorithm (i.e., Algorithm~\ref{alg01}). In
section~\ref{secPEVPs}, we generalize our approach to polynomial
eigenvalue problems (PEVPs) and compare the resulting algorithm to
other polynomial eigensolvers.
In all our numerical experiments, we observed that
Assumption~\ref{ass01} holds.

\subsection{Scalar polynomials}\label{secSCALAR}
The backward error measured in all the experiments of this section is the
upper bound  $\berrc$ in \eqref{ineq.upperbnd}
on the min-max elementwise backward error.

\begin{exemple}

\item\label{exp1}

We take $100$ samples of a polynomial of degree $50$. Each polynomial has random
complex zeros generated as follows:
the multiplicity is $1$, the modulus is $10^e$ with $e$ uniformly random between
$-20$ and $+20$, and the argument is uniformly random between $0$ and $2 \pi$.
Figure~\ref{fig961}(a) shows the backward errors (less than $10^{-4}$) for the
zeros returned by \texttt{roots} and by the new algorithm while
Figure~\ref{fig961}(b) also shows the larger backward errors.
\begin{figure}
\begin{center}
\subfigure[]{\includegraphics[scale=0.37,trim={0.9cm 0.0cm 0.9cm 0.7cm},clip]{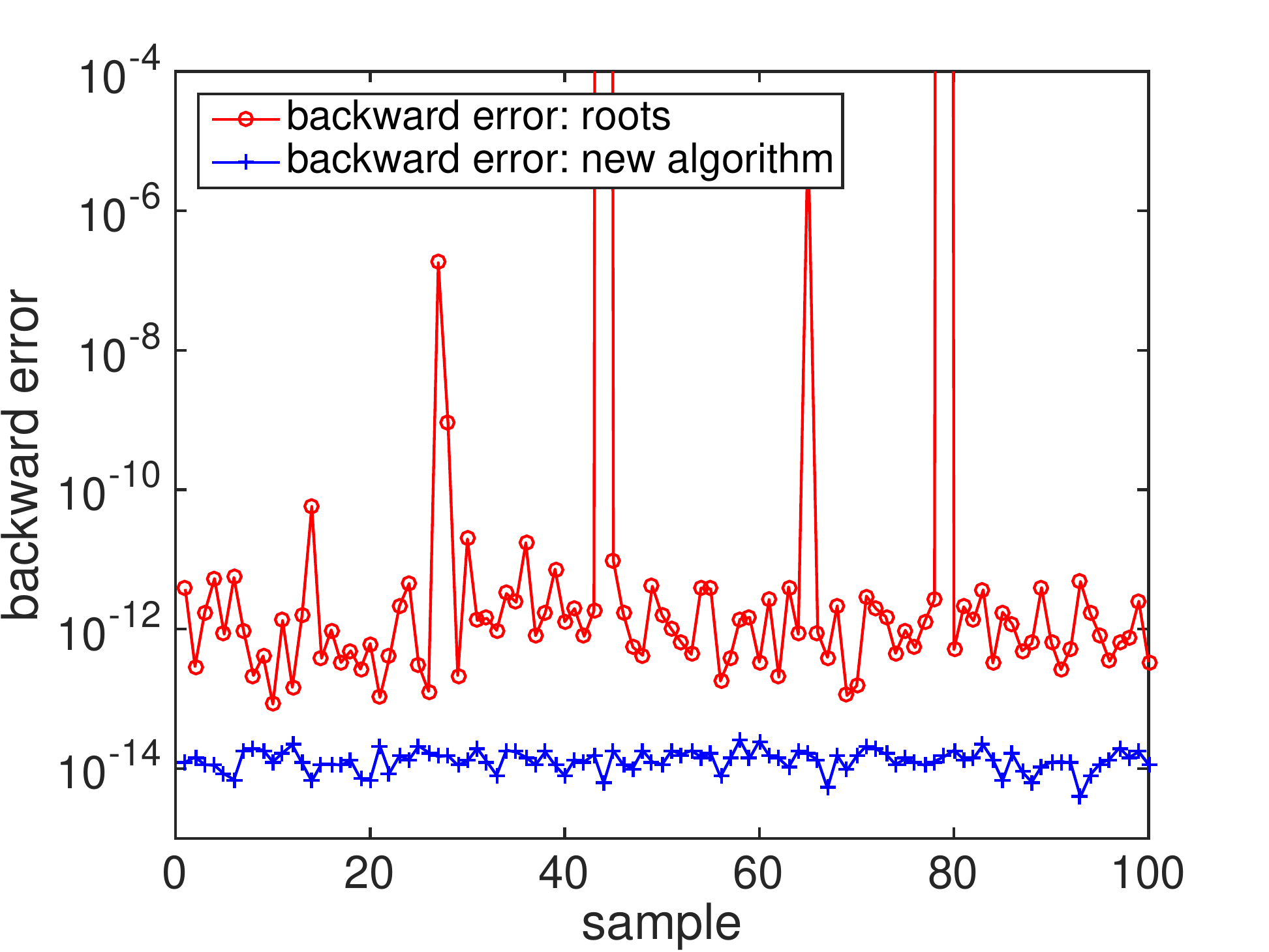}}
\subfigure[]{\includegraphics[scale=0.37,trim={0.9cm 0.0cm 0.9cm 0.7cm},clip]{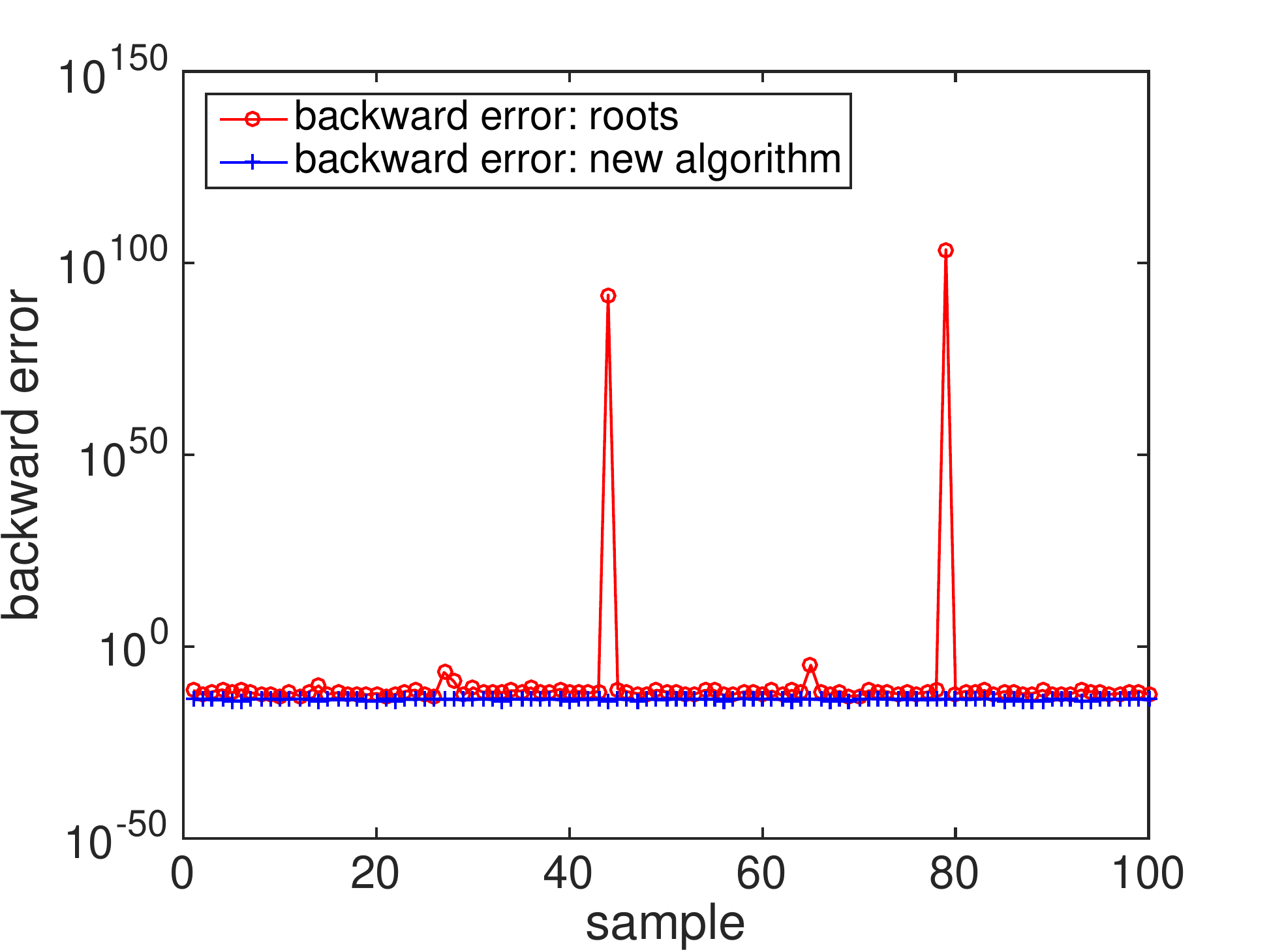}}
\caption{Min-max backward errors for the zeros computed by
the MATLAB \texttt{roots} function and the new algorithm for
Experiment~\ref{exp1}.\label{fig961}}
\end{center}
\end{figure}

For sample number $44$,  we compare
in Figure~\ref{fig963} the modulus of the coefficients of $p(z)$ to the
modulus of the coefficients of the polynomial
$\widetilde{p}(z) = p_d\prod_{k=1}^{50} (z-\hz_k)$ constructed from the zeros
$\hz_k$, $k=1,\ldots, 50$ returned by \texttt{roots} and by the new algorithm.
The zeros returned by \texttt{roots} do not reproduce the first coefficients of
the polynomial with a high relative accuracy leading to a large backward error
as shown in Figure~\ref{fig961}(b).

\begin{figure}
\centering
\includegraphics[scale=0.45]{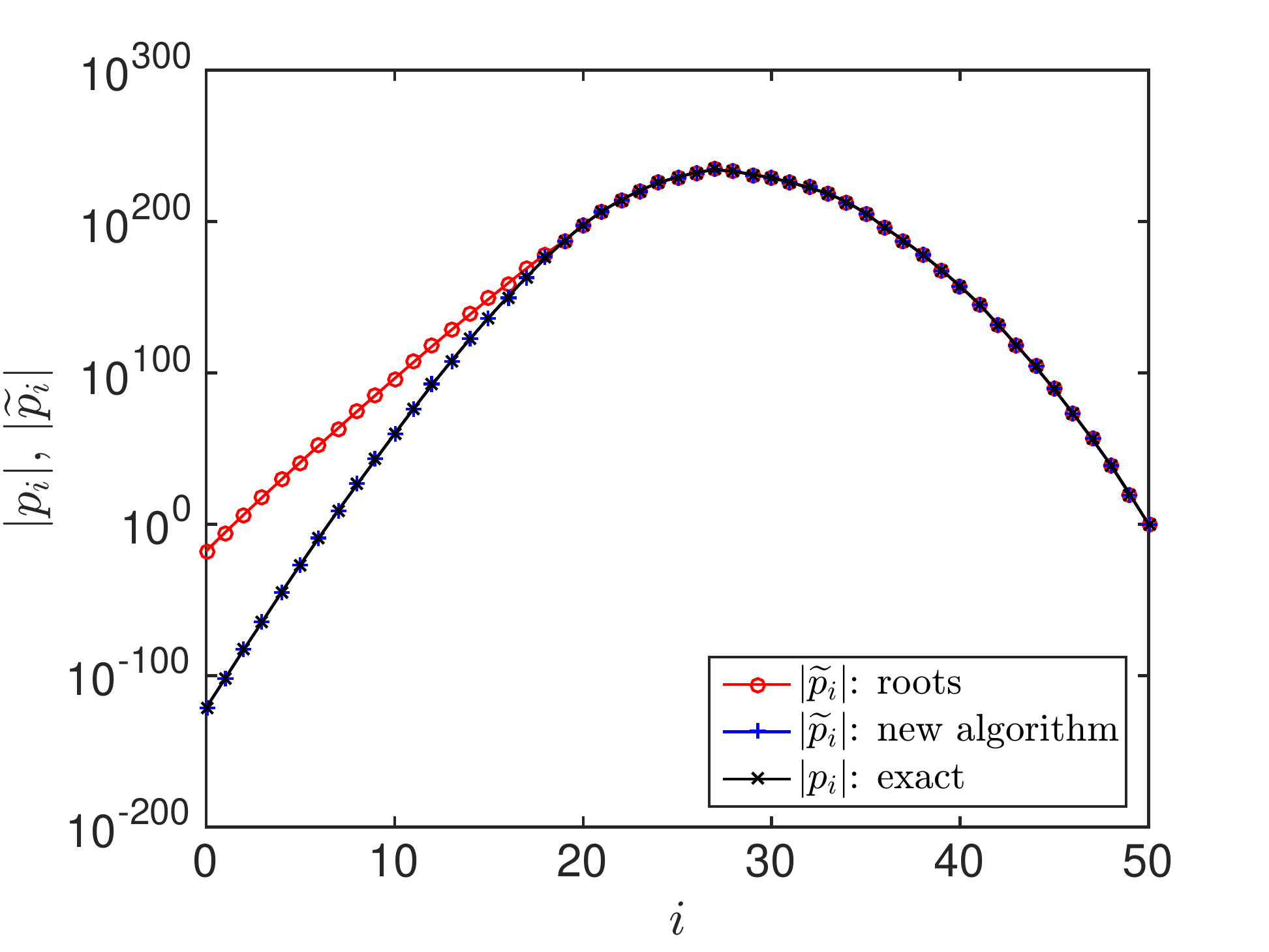}
\caption{Plot of $|p_i|$ and $|\widetilde{p}_i|$, where
$\widetilde{p}(z) = p_d\prod_{k=1}^{50} (z-\hz_k)$ is constructed from the zeros
$\hz_1,\ldots,\hz_{50}$ computed by \texttt{roots} and by the new algorithm
for sample 44 of Experiment~\ref{exp1}.\label{fig963}}
\end{figure}

\item\label{exp2}
We generate $100$ polynomials of degree $30$.
Each polynomial has random complex zeros computed as follows:
the multiplicity is uniformly random between $1$ and $30$, the modulus is $10^e$ with $e$ uniformly random between $-10$ and $+10$,
and the argument is uniformly random between $0$ and $2 \pi$.
Figure~\ref{fig964} compares the min-max backward error between {\tt roots}
and our new algorithm.
\begin{figure}
\centering
\subfigure[]{\includegraphics[scale=0.37,trim={0.9cm 0.0cm 0.9cm 0.7cm},clip]{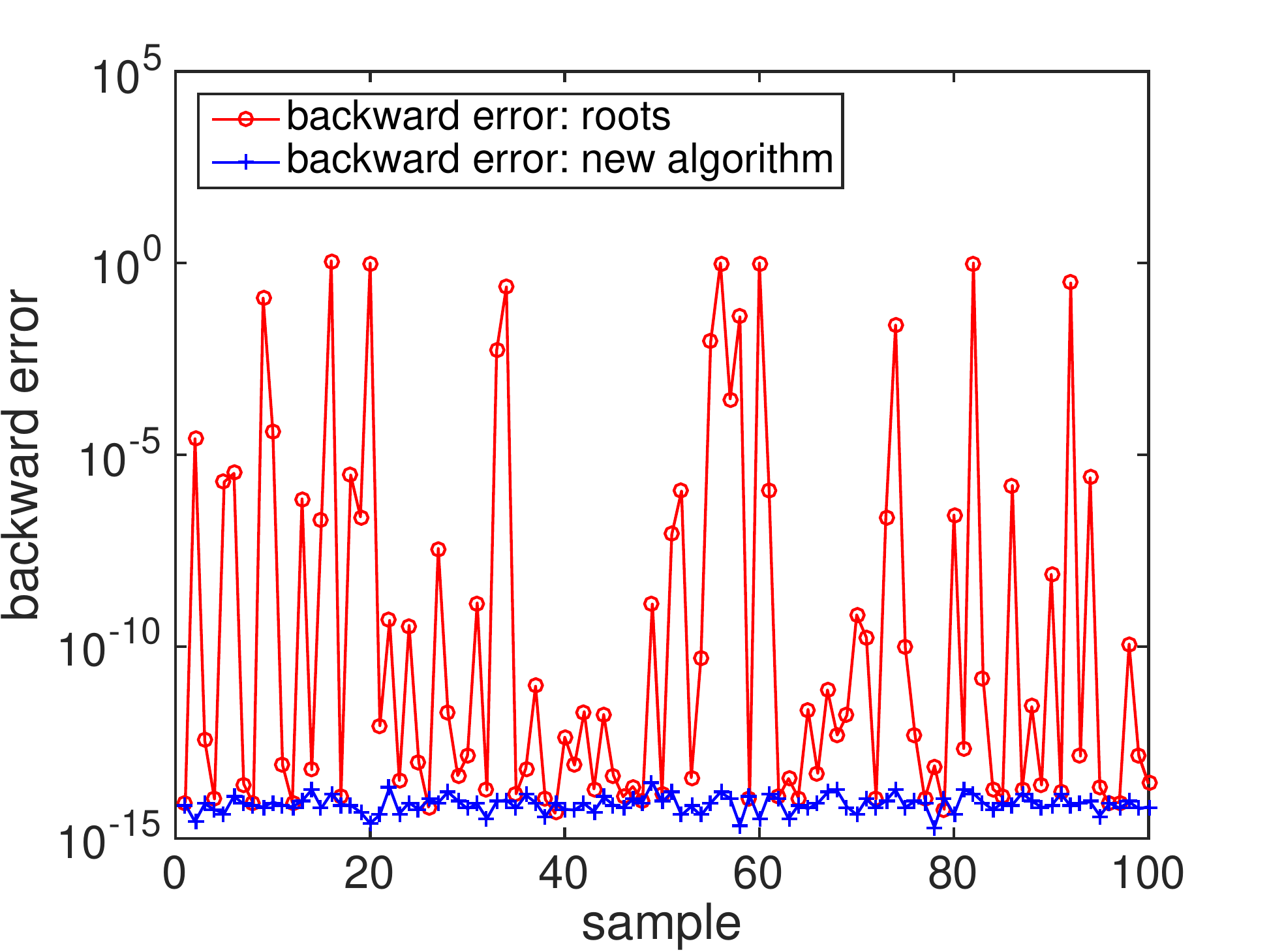}}
\subfigure[]{\includegraphics[scale=0.37,trim={0.9cm 0.0cm 0.9cm 0.7cm},clip]{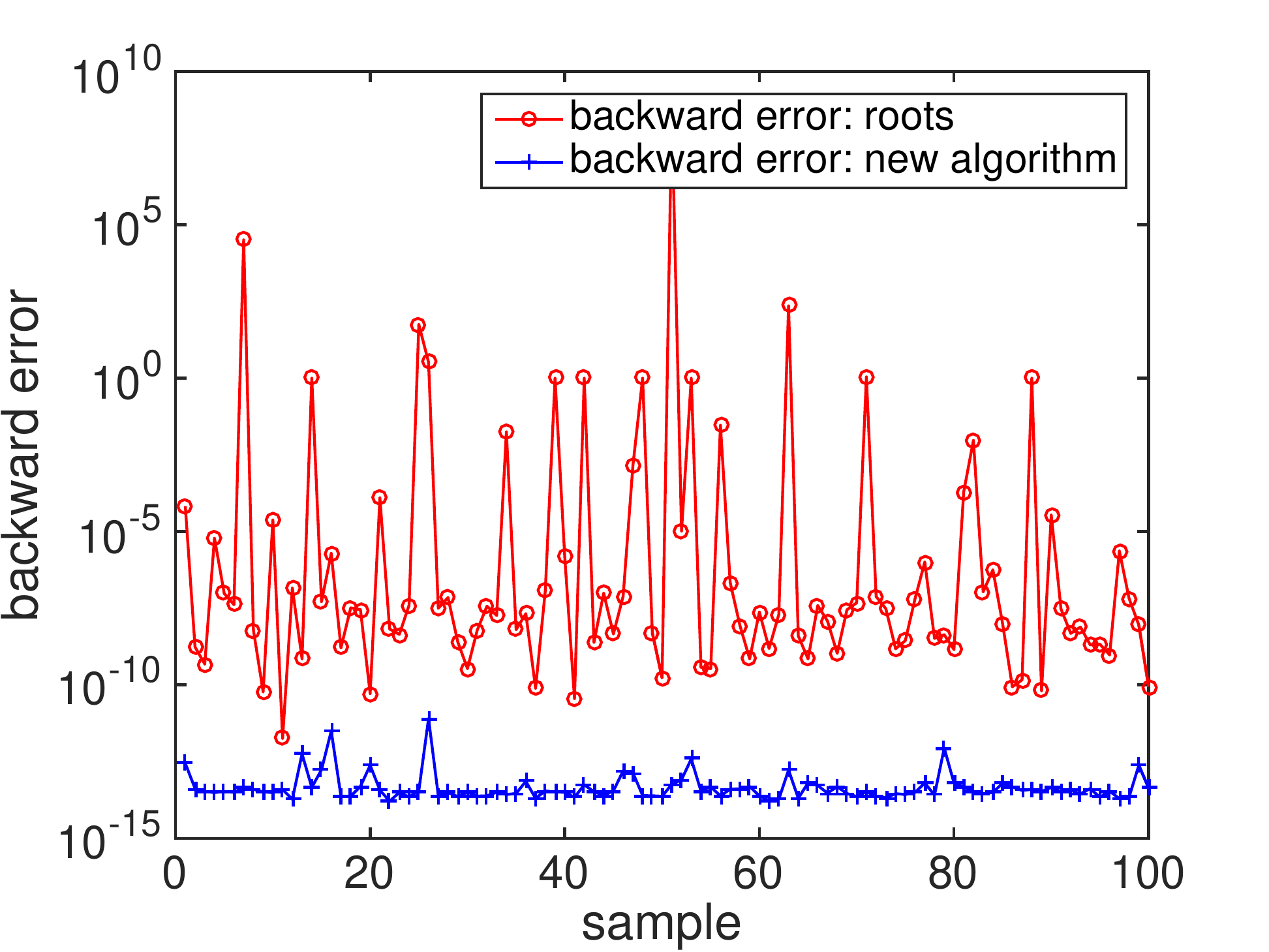}}
\caption{Min-max backward error for the zeros computed by
\MATLAB's \texttt{roots} function and the new algorithm for
Experiment~\ref{exp2} in (a) and for Experiment~\ref{exp3} in (b).\label{fig964}}
\end{figure}

\item\label{exp3}
We take $100$ samples of a polynomial of degree $100$.
Each polynomial has random complex coefficients as follows.
The modulus is $10^e$ with $e$ uniformly random between $-20$ and $+20$
and the argument is uniformly random between $0$ and $2 \pi$.
Figure~\ref{fig964}(b) shows the backward error.

\item\label{exp4}
The parameters are the same as in Experiment~3 but now the degree is $20$
instead of $100$. Figure~\ref{fig965}(a) shows the backward error.
\begin{figure}
\begin{center}
\subfigure[]{\includegraphics[scale=0.37,trim={0.9cm 0.0cm 0.9cm 0.7cm},clip]
{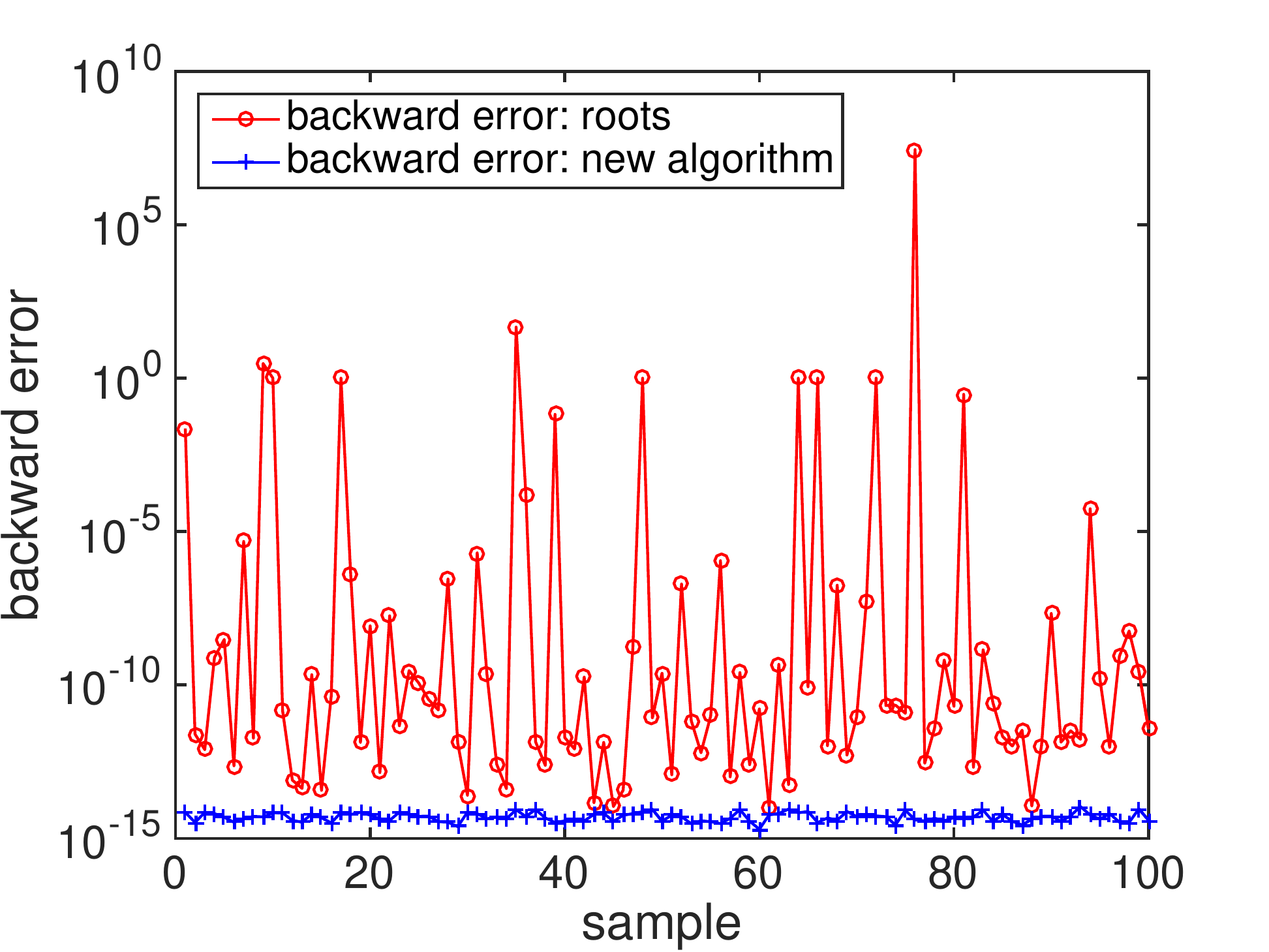}}
\subfigure[]{\includegraphics[scale=0.37,trim={0.9cm 0.0cm 0.9cm 0.7cm},clip]{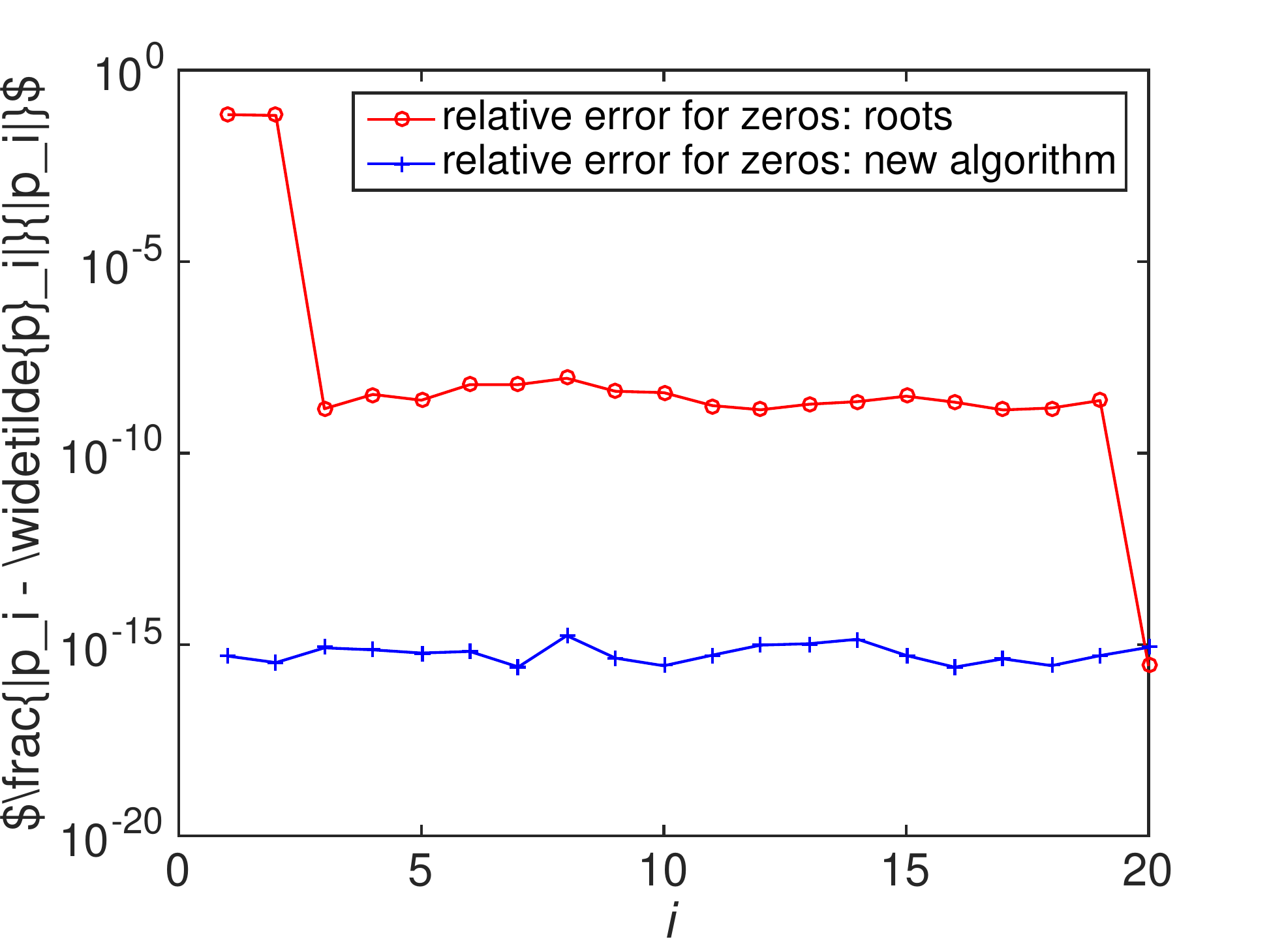}}
\caption{Plot (a): min-max backward error for the zeros computed by
\MATLAB's \texttt{roots} function and the new algorithm for
Experiment~\ref{exp4}.
Plot (b):  relative error on the computed zeros by \texttt{roots} and the new
algorithm for sample 39 of Experiment~\ref{exp4}.\label{fig965}}
\end{center}
\end{figure}
Figure~\ref{fig965}(b) compares the relative errors on the computed zeros by
{\tt roots} and the new algorithm for sample 39.
For this sample, Figure~\ref{fig966} with plot (a) for the new algorithm and
plot (b) for \texttt{roots}, shows the magnitude of the coefficients of $p(z)$
and $\widetilde{p}(z)$ with $\widetilde{p}_\deg=p_\deg$, the absolute errors
$|p_i-\widetilde{p}_i|$ and compare them to the convex hull
of the set of points $(i,\log|p_i|)$, $i = 0,1,\ldots,\deg$
as well as the points on this upper boundary multiplied by the machine precision
$\epsmach$. For our new algorithm the absolute error is not much larger than
$\epsmach$ times the convex hull indicating that the corresponding backward
error $\berrc$ is of the size of the machine precision $\epsmach$.
This is not the case for \texttt{roots}, in particular,
the absolute error on the coefficient $p_1$ is almost as large as the coefficient
itself.
This indicates that the backward error $\berrc$ is of order $1$.

\begin{figure}
\begin{center}
\subfigure[]{\includegraphics[scale=0.36,trim={0.3cm 0.0cm 0.9cm 0.4cm},clip]
{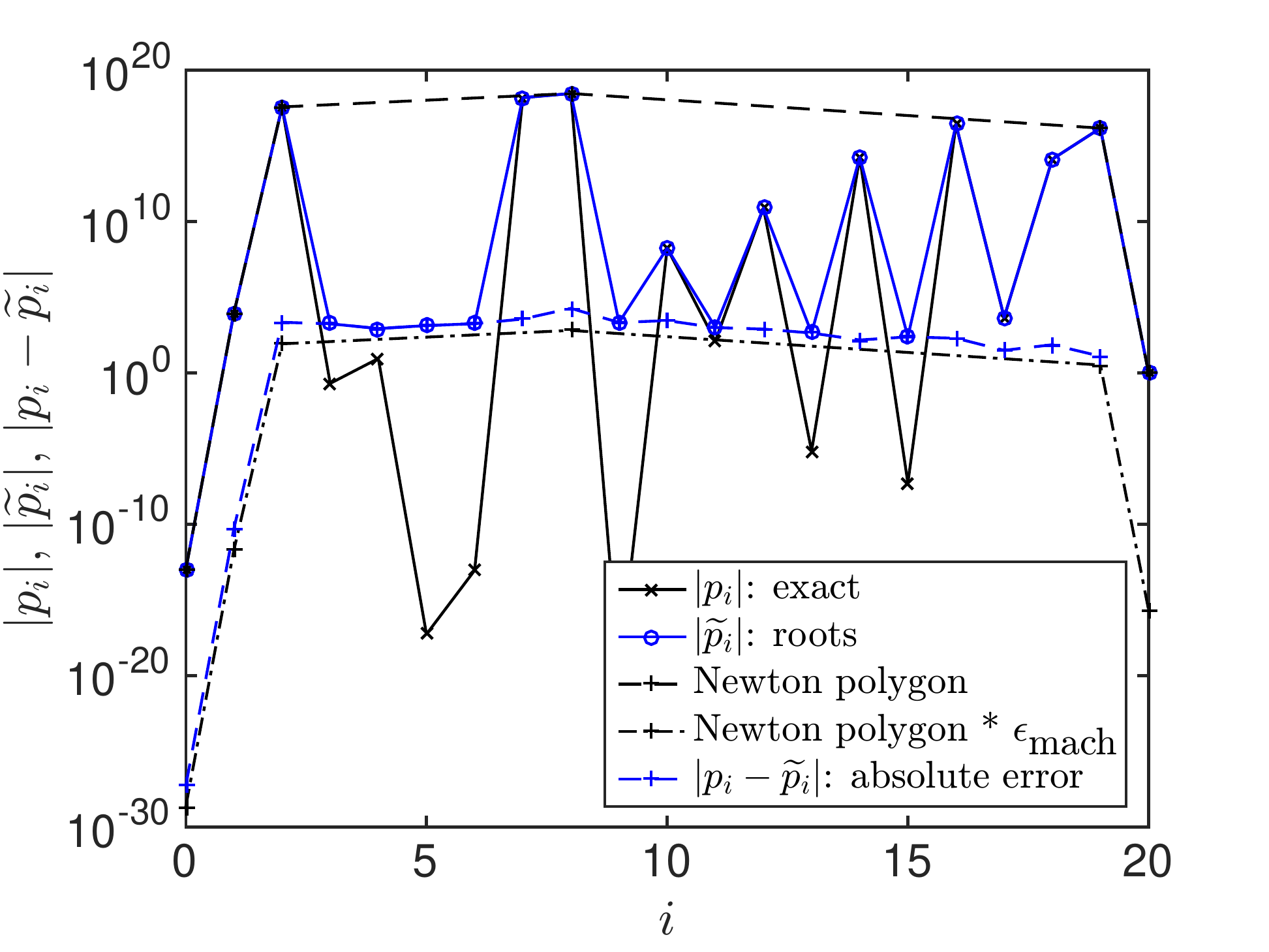}}
\subfigure[]{\includegraphics[scale=0.36,trim={0.3cm 0.0cm 0.9cm 0.4cm},clip]
{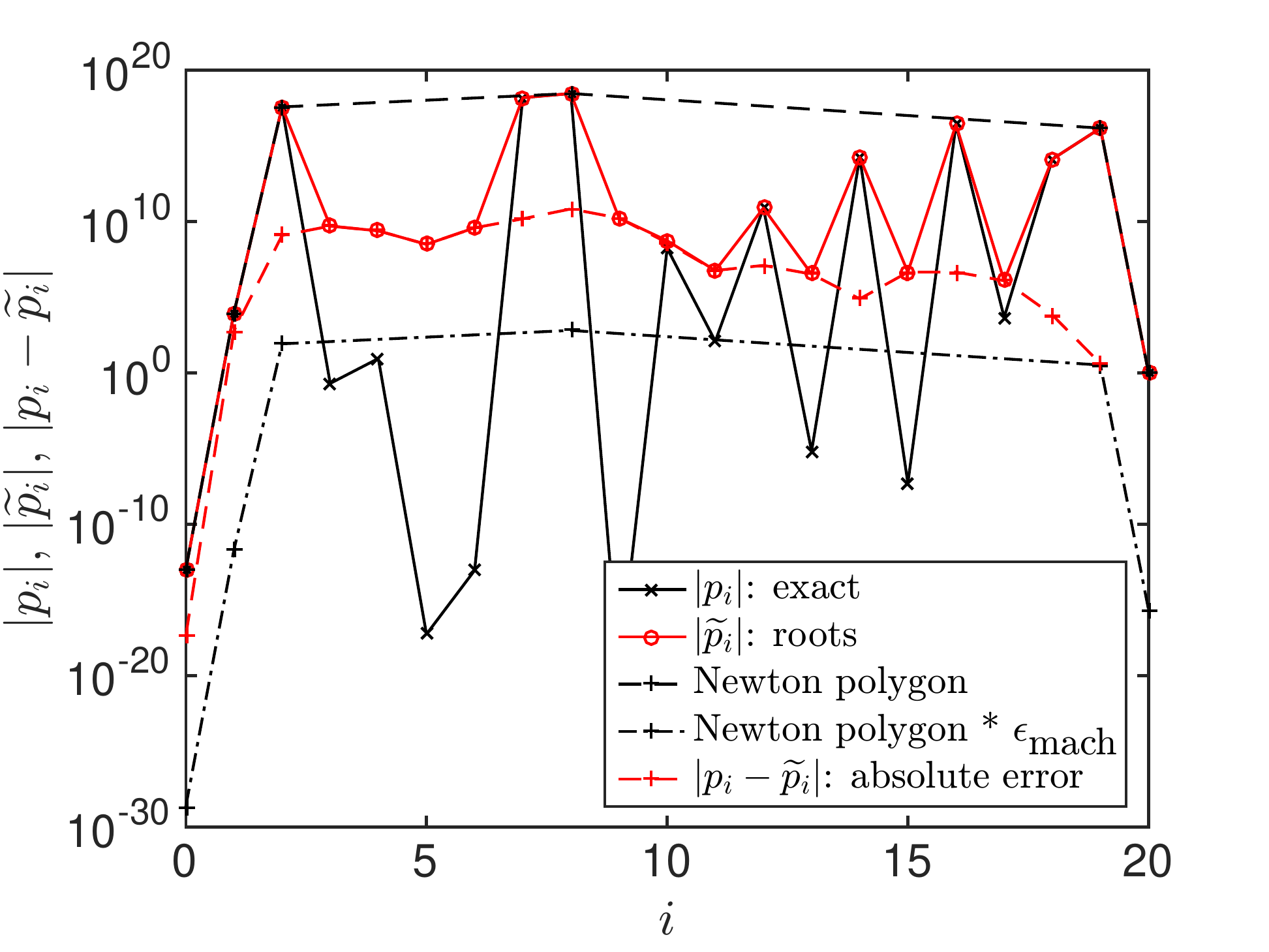}}
\caption{Sample 39 of Experiment~\ref{exp4}.
Plot (a) corresponds to zeros computed by the new algorithm, whereas plot (b) is
for zeros computed by \texttt{roots}.
The figures show the coefficients $|p_i|$ of $p(z)$,
the coefficients $|\widetilde{p}_i|$ of $\widetilde{p}(z)$
as well as the absolute error $|p_i - \widetilde{p}_i|$.
The Newton polygon of the points $(i,\log|p_i|)$, $i = 0,1,\ldots,\deg$ is shown
as well as the points on this polygon shifted down by a factor $\epsmach$.
\label{fig966}}
\end{center}
\end{figure}



\end{exemple}

\subsection{Generalization to polynomial eigenvalue problems}\label{secPEVPs}
Given a matrix polynomial $P(z)=\sum_{i=0}^\deg P_i z^i\in \CC[z]^{s\times s}$,
the polynomial eigenvalue problem (PEVP) consists of finding
scalars $\lambda$ (eigenvalues) and corresponding nonzero vectors $v$
(eigenvectors) such that
$$ P(\lambda) v = 0. $$
Algorithm~\ref{alg01} extends easily from scalar polynomial to matrix polynomial.
The entries in the companion form~\eqref{def.compa} are replaced by matrices
(i.e., $p_i$ is replaced by $P_i$, $1$ by $I_s$ and $0$ by the $s\times s$
identity matrix) to obtain a $\deg s\times \deg s$
block companion linearization $C(\l)$
for the grade $\deg+1$ matrix polynomial $0z^{\deg+1}+P(z)$.
For the two-sided diagonal scaling,
we use $D_l\otimes I_s$ and $D_r\otimes I_s$ with $D_l$ and $D_r$ as in
\eqref{def.Dl}--\eqref{def.Dr}, and $\ttr_i$ as in \eqref{def.ttr}.
The positive scalars $\tr_i$, $i=1,\ldots,t$ with $\tr_i$ of multiplicity
$m_i$ are the tropical roots of $\tpol(x) = \max_i \|P_i\| x^i$.
The resulting block pencil $\hC(\l) = (D_l\otimes I_s)C(\l)(D_r\otimes I_s)
=\hA -\l \hB$ is such that $\hA$ is well-balanced in the sense that the $s\times
s$ matrices in the first block row of $\hA$ have norms less or equal to $1$, and
$\hB$ is graded.
The deflation of the $s$ extra eigenvalues at
infinity is performed by constructing a QR factorization of the first block
column of $\hC(\l)$ and by forming $Q^*\hC(\l)$.
We can deflate the first $s$ rows and columns of the resulting pencils and call
the QZ algorithm together with the strong deflation criterion for eigenvalues at
infinity we discussed in section~\ref{sec_algo}.

We consider the following polynomial eigensolvers:

\begin{enumerate}

\item the MATLAB \polyeig function;

\item \quadeig from \cite{hmt13} when the degree $\deg=2$;

\item Gaubert and Sharify's Algorithm \cite[Alg. 1]{r957} (see also \cite[Alg.
4.1]{r953}). We use the same MATLAB implementation as in \cite{r953}, which we
refer to as the \texttt{G\&S} eigensolver.

\item the polynomial eigensolver based on a tropically scaled Lagrange linearization
using well-separated tropical roots described in \cite{VBT2017},
which we refer to as the \texttt{Lagrange} eigensolver;

\item the eigensolver based on scaled block companion pencil as described at the
start of section~\ref{secPEVPs}, which we refer to as the \texttt{new} eigensolver.
\end{enumerate}

The normwise backward error for an approximate eigenvalue $\tilde{\l}$ of $P$
can be computed as \cite{q654}
\begin{equation}\label{eq.berrlambda}
\eta_P(\tilde{\l})=\frac{\| P(\tilde{\lambda})^{-1} \|_2^{-1}}{\sum_{i=0}^d |\tilde{\lambda}|^i \|P_i\|_2}
= \frac{\sigma_{\min} (P(\tilde{\lambda}))}{\sum_{i=0}^d |\tilde{\lambda}|^i \|P_i\|_2} .
\end{equation}
This backward error is the smallest $\eps$ such that $\tilde{\l}$ is an
eigenvalues of $P(\l)+\Delta P(\l)$ with
$\Delta P(z) = \sum_{i=0}^\deg \Delta P_i z^i$ such that
$\|\Delta P_i\|\le \epsilon\|P_i\|$, $i=0,\ldots,\deg$.
For a min-max normwise backward error, it is sufficient to replace
$\sum_{i} |\tilde{\lambda}|^i \|P_i\|_2$ in \eqref{eq.berrlambda} with
$\max_{i} |\tilde{\lambda}|^i \|P_i\|_2$.
As was shown in~\eqref{bndberr} for a single approximate zero of a scalar polynomial, there is not much
difference between these two measures of the backward error.
Note that we are not looking at a global measure of the backward error here
but, instead, report
$$
        \eta_P^{\max} = \max\{\eta_P(\tilde \l):
        \mbox{$\tilde \l$ is an eigenvalue of $P$}\}
$$
which is a lower bound of the global backward error for all the computed
eigenpairs of $P$.
We consider that all the eigenvalues have been computed with a small backward
error if $\eta_P^{\max}\le \deg s \emach$, where for our numerical experiments
$\emach\approx 2.2\times 10^{-16}$.


\begin{exemple}
\setcounter{exemple}{4}
\item\label{exp5}
We consider all square problems from the NLEVP collection \cite{bhms13}
with size $s\le 300$ and with $\eta_P^{\max}\le s \emach$.
The value of $\eta_P^{\max}$ is displayed in
Table~\ref{tab100} for each polynomial eigensolver under consideration.
A backward error $\eta_P^{\max}$ larger than $\deg s\emach$ is highlighted in
red and bold. The eigensolvers \texttt{G\&S}, \texttt{Lagrange}, and \texttt{new}
return eigenvalues with small backward errors for almost all the problems as opposed to
\texttt{polyeig}. For the \texttt{cd\_player} problem,
the \texttt{G\&S} eigensolver returns eigenvalues with a large backward error,
whereas $\eta_P^{\max}\approx \deg s\emach$ for the \texttt{relative\_pose\_5pt}
problem when solved by \texttt{Lagrange} and for the \texttt{plasma\_drift}
problem when solved by \texttt{new}.
\begin{table}
\caption{
Largest backward errors $\eta_P^{\max}$ for eigenvalues
computed by the eigensolvers \texttt{polyeig}, \texttt{quadeig} (for quadratics
only), \texttt{G\&S}, \texttt{Lagrange} and \texttt{new} on test problems
from the NLEVP collection as described in Experiment~\ref{exp5}.\label{tab100}}
\begin{center}
\begin{tabular}{rrrrrrrrr}
\hline
Problem & $d$ & $s$ & \tt{polyeig} &\tt{quadeig} &
\tt{G\&S} & \tt{Lagrange} & \tt{new} \\
\hline
\ttb{           cd\_player} &   2 &  60 &   \tc{3.1e-10} &   2.5e-16 &   \tc{7.5e-07} &   4.1e-16 &   1.4e-15 \\
\ttb{         damped\_beam} &   2 & 200 &   \tc{2.6e-11} &   2.9e-16 &   1.2e-16 &   4.8e-16 &   6.9e-17 \\
\ttb{            hospital} &   2 &  24 &   \tc{2.9e-13} &   1.3e-15 &   1.6e-15 &   3.9e-15 &   2.7e-15 \\
\ttb{         metal\_strip} &   2 &   9 &   \tc{4.1e-14} &   6.8e-16 &   2.7e-16 &   3.0e-16 &   3.5e-16 \\
\ttb{              mirror} &   4 &   9 &   \tc{2.1e-14} &   \na &   3.7e-17 &   9.8e-16 &   5.4e-17 \\
\ttb{      orr\_sommerfeld} &   4 &  64 &   \tc{9.1e-08} &   \na &   7.1e-15 &   1.5e-15 &   1.4e-15 \\
\ttb{      pdde\_stability} &   2 & 225 &   \tc{1.6e-13} &   {4.0e-14} &   {1.4e-14} &   {8.8e-14} & {9.1e-14} \\
\ttb{    planar\_waveguide} &   4 & 129 &   \tc{4.7e-12} &   \na &   {3.2e-14} &   2.7e-15 &   {1.8e-14} \\
\ttb{        plasma\_drift} &   3 & 128 &   \tc{2.2e-13} &   \na &  {1.3e-14} &   {1.6e-14} &   \tc{1.0e-13} \\
\ttb{         power\_plant} &   2 &   8 &   \tc{5.3e-12} &   4.2e-18 &   3.3e-18 &   1.3e-16 &   3.1e-18 \\
\ttb{   relative\_pose\_5pt} &   3 &  10 &   9.6e-18 &   \na &   1.1e-16 &   \tc{2.1e-14} &   8.5e-17 \\
\ttb{         speaker\_box} &   2 & 107 &   \tc{1.7e-13} &   6.0e-17 &   4.1e-17 &   6.8e-16 &   8.2e-18 \\
\ttb{            wiresaw1} &   2 &  10 &   \tc{1.3e-14} &   1.4e-15 &   1.0e-15 &   1.0e-15 &   1.8e-15 \\
\ttb{            wiresaw2} &   2 &  10 &   \tc{2.0e-14} &   1.9e-15 &   1.5e-15 &   9.7e-16 &   8.3e-16 \\
\hline
\end{tabular}
\end{center}
\end{table}

\item\label{exp6}
In \cite{VBT2017}, we considered several PEVPs with large variations
in the magnitude of their eigenvalues (and, hence, also in norm of their
matrix coefficients).
The backward errors for these problems are provided in Table~\ref{tab101}, and
the backward errors for which $\eta_P^{\max}\ge\deg s\emach$ are
highlighted in red and bold.
The \texttt{Lagrange} and \texttt{new} eigensolvers return eigenvalues with a
small backward error for almost all the problems, the backward errors
highlighted in red for these two eigensolvers being just slightly larger $\deg
s\emach$ (an exception being \texttt{Lagrange} with \texttt{Problem 17}).
%
%
\begin{table}
\caption{
Largest backward errors $\eta_P^{\max}$ for eigenvalues
computed by the eigensolvers \texttt{polyeig}, \texttt{quadeig} (for quadratics),
\texttt{G\&S}, \texttt{Lagrange} and \texttt{new} on test problems
used in \cite{VBT2017}.\label{tab101}}
\begin{center}
\begin{tabular}{rrrrrrrrr}
\hline
Problem & $d$ & $s$ & \tt{polyeig} &\tt{quadeig} &
\tt{G\&S}& \tt{Lagrange} & \tt{new} \\
\hline
\tt{           Problem 1} &   7 &   4 &   \tc{3.0e-02} &   \na &   \tc{1.7e-08} &   1.4e-15 &   6.3e-16 \\
\tt{           Problem 2} &   7 &   4 &   \tc{2.7e-01} &   \na &   \tc{6.8e-13} &   2.1e-15 &   5.7e-16 \\
\tt{           Problem 3} &   2 &   4 &   2.3e-16 &   2.2e-16 &   1.1e-16 &\tc{3.2e-15} &   1.4e-16 \\
\tt{           Problem 4} &   2 &   5 &   2.2e-16 &   \tc{1.3e-11} &   1.9e-16 &   3.0e-16 &   2.3e-16 \\
\tt{           Problem 5} &   2 &   5 &   4.7e-16 &   \tc{2.8e-13} &   1.0e-16 &   4.1e-16 &   3.0e-16 \\
\tt{           Problem 6} &   2 &   2 &   3.2e-17 &   3.4e-17 &   4.0e-18 &   2.8e-16 &   4.5e-18 \\
\tt{           Problem 7} &   2 &  10 &   2.1e-16 &   \tc{1.3e-02} &   1.4e-16 &   3.2e-16 &   2.1e-16 \\
\tt{           Problem 8} &   2 &  10 &   4.9e-15 &   \tc{3.4e-12} &   3.1e-16 &   3.0e-16 &   7.8e-16 \\
\tt{           Problem 9} &   2 &  40 &   \tc{6.8e-07} &   2.1e-15 &   4.0e-16 &   4.0e-16 &   2.5e-16 \\
\tt{          Problem 10} &   5 &  20 &   \tc{3.1e-12} &   \na &   1.4e-15 &   1.4e-15 &   1.1e-15 \\
\tt{          Problem 11} &  10 &   8 &   \tc{2.9e-09} &   \na &   \tc{1.6e-13} &   2.3e-15 &   1.9e-15 \\
\tt{          Problem 12} &   4 &  30 &   \tc{2.2e-11} &   \na &   \tc{5.5e-14} &   9.0e-16 &   6.5e-15 \\
\tt{          Problem 13} &   4 &   9 &   \tc{4.6e-12} &   \na &   1.8e-15 &   \tc{1.1e-14} &   \tc{1.2e-14} \\
\tt{          Problem 14} &   4 &  64 &   \tc{9.1e-08} &   \na &   7.1e-15 &   1.5e-15 &   1.4e-15 \\
\tt{          Problem 17} &  10 &   2 &   \tc{3.2e-01} &   \na &   \tc{2.8e-01} &   \tc{8.1e-13} &   8.8e-16 \\
\tt{          Problem 18} &   4 &   4 &   \tc{1.5e-11} &   \na &   \tc{2.3e-14} &   6.9e-16 &   3.7e-16 \\
\tt{          Problem 19} &   4 &   4 &   \tc{7.8e-13} &   \na &   \tc{4.4e-14} &   7.9e-16 &   8.7e-17 \\
\tt{          Problem 20} &   5 &   4 &   \tc{8.7e-03} &   \na &   \tc{1.3e-06} &   1.1e-15 &   1.3e-15 \\
\tt{          Problem 21} &   5 &   4 &   \tc{3.1e-07} &   \na &   \tc{4.7e-07} &   8.9e-16 &   3.4e-16 \\
\tt{          Problem 22} &   4 &   4 &   \tc{7.2e-08} &   \na &   \tc{5.5e-12} &   8.4e-16 &   2.2e-16 \\
\hline
\end{tabular}
\end{center}
\end{table}

\end{exemple}


\section{Conclusions}\label{sec_concl}
We introduced a new measure of the backward error for roots of scalar
polynomials that is less strict than the elementwise relative backward error but
is still meaningful.
This new measure  allows larger perturbations on the coefficients that do not
participate much in the backward error.
For this we used an associated max-times polynomial and its tropical roots to determine
how much each coefficient can be perturbed.
We showed with examples of scalar polynomials with well conditioned zeros that
our new backward error can provide an upper bound on the forward error
that is sharper than the elementwise relative backward error.

We designed a new algorithm for computing the zeros of scalar polynomials as
well as the eigenvalues of matrix polynomials.
Our algorithm is based on a companion linearization $A-zB$ of the (matrix) polynomial
to which we artificially added a zero leading (matrix) coefficient.
In doing so, we found that we could construct a two-sided diagonal scaling that
balances
$A$ and transforms $B$
into a graded matrix.
We observed in \cite{VBT2017} that if we
use an implementation of the QZ algorithm with a strict deflation criterion for
the eigenvalues at infinity on such scaled pencils, then
the backward error on the scaled pencil has certain properties:
it is of the order of the machine precision for the scaled matrix $A$
and graded for the scaled matrix $B$.
So under the assumption that this observation holds,
we proved that our new polynomial root finder
is backward stable with respect to the newly defined backward error.
Several numerical experiments show the
stability of this approach for approximating the zeros of scalar polynomials as
well as the eigenvalues of matrix polynomials.


\bibliographystyle{abbrv}
\bibliography{strings,ft}

\end{document}